\pdfoutput=1 
\documentclass[a4paper,12pt,twoside,onecolumn,final]{article}
\usepackage{amsmath,amsfonts,amssymb,amscd,mathrsfs}
\usepackage{extarrows}
\usepackage{epsfig}
\usepackage{enumerate}
\usepackage{color}
\usepackage{graphicx}
\usepackage{psfrag}
\usepackage[nodayofweek]{datetime}
\usepackage[paper=a4paper,hscale=0.7,vscale=0.765,centering,heightrounded]{geometry}
\usepackage[ntheorem]{empheq} 
\empheqset{box=\bigfbox} 
\usepackage[thmmarks,amsmath,hyperref]{ntheorem} 
\theoremstyle{plain} 
\theoremheaderfont{\normalfont\bfseries} 
\theorembodyfont{\slshape}
\theoremsymbol{\ensuremath{_\Box}}
\theoremnumbering{Alph}
\newtheorem{theorem}{Theorem}[section]

\theoremnumbering{arabic}
\newtheorem{proposition}{Proposition}[section]
\newtheorem{lemma}[proposition]{Lemma}

\theoremstyle{plain}
\theoremheaderfont{\normalfont\bfseries} 
\theorembodyfont{\upshape}
\theoremsymbol{\ensuremath{_\Box}}
\newtheorem{definition}[proposition]{Definition}  
\newtheorem{remark}[proposition]{Remark}  
\newtheorem{example}[proposition]{Example} 
\newtheorem{assumption}[proposition]{Assumption}  

\theoremstyle{nonumberplain}
\theoremheaderfont{\normalfont\bfseries} 
\theorembodyfont{\upshape}
\theoremsymbol{\ensuremath{_\blacksquare}} 
\newtheorem{proof}{Proof}
\theoremsymbol{\ensuremath{_\blacksquare}}

\theoremlisttype{allname}
\allowdisplaybreaks[3] 
\usepackage[colorlinks=true, pdfstartview=FitV, linkcolor=redgray, citecolor=bluegray, urlcolor=black, pdfpagelabels, naturalnames=true, draft=false]{hyperref}
\usepackage[small,bf]{caption} 
\usepackage[notref,notcite,draft]{showkeys}

\definecolor{refkey}{rgb}{0,1,0} 
\definecolor{labelkey}{rgb}{0,1,1}
  \renewenvironment{thebibliography}[1]{%
    \begin{oldthebibliography}{#1}%
       \addtolength{\itemsep}{-0.5ex}%
  }%
  {%
    \end{oldthebibliography}%
  }

\graphicspath{{Figures/}{./}}

\usepackage{xfrac}

\usepackage{psfrag}
\usepackage[off]{auto-pst-pdf}

\usepackage{tikz}
\usetikzlibrary{shapes.multipart}
\usepackage{contour}

\DeclareMathAlphabet{\mathpzc}{OT1}{pzc}{m}{it}

\usepackage{bbm}
\usepackage[ulem=normalem,deletedmarkup=sout]{changes}
\newcommand\redout{\bgroup\markoverwith
{\textcolor{red}{\rule[.5ex]{2pt}{0.75pt}}}\ULon}

\setdeletedmarkup{\textcolor{red}{\redout{#1}}}

\newcommand\soutmath[2][red]{\setbox0=\hbox{$\,#2\,$}%
\rlap{\raisebox{.325\ht0}{\textcolor{#1}{\rule{\wd0}{.75pt}}}}\textcolor{red}{\,#2\,}} 

\newcommand\souteq[2][red]{\setbox0=\hbox{$#2$}%
$$\rlap{\raisebox{.325\ht0}{\textcolor{#1}{\rule{\wd0}{.75pt}}}}\textcolor{red}{\text{$#2$}}$$}

\usepackage{ifdraft}
\ifdraft{\newcommand{\removedmath}[1]{\soutmath{#1}}}{\newcommand{\removedmath}[1]{{}}}
\ifdraft{\newcommand{\removedeq}[1]{\souteq{#1}}}{\newcommand{\removedeq}[1]{{}}}
\ifdraft{\newcommand{\removed}[1]{\deleted{#1}}}{\newcommand{\removed}[1]{}}

\ifdraft{}{}
\ifdraft{}{}

\usepackage{marginnote}

\numberwithin{equation}{section}

\newcommand{\mbbB}{\mathbb{B}}

\newcommand{\mbbP}{\mathbb{P}}

\newcommand{\mbbR}{\mathbb{R}}
\newcommand{\mbbS}{\mathbb{S}}
\newcommand{\mbbU}{\mathbb{U}}
\newcommand{\mbbV}{\mathbb{V}}
\newcommand{\mbbW}{\mathbb{W}}

\newcommand{\CAO}{C_{\text{\tiny$\mathrm{AO}$}}}
\newcommand{\CBM}{C_{\text{\tiny$\mathrm{BM}$}}}
\newcommand{\CPF}{C_{\text{\tiny$\mathrm{PF}$}}}

\newcommand{\Grho}{W^\rho(\bsbeta;\Omega)}
\newcommand{\Gsigma}{W^\sigma(\bsbeta;\Omega)}
\newcommand{\Grhominus}{W^\rho_{0,-}(\bsbeta;\Omega)}
\newcommand{\Grhoplus}{W^\rho_{0,+}(\bsbeta;\Omega)}
\newcommand{\Grhopm}{W^\rho_{0,\pm}(\bsbeta;\Omega)}
\newcommand{\Gp}{W^p(\bsbeta;\Omega)}
\newcommand{\Gpminus}{W^p_{0,-}(\bsbeta;\Omega)}
\newcommand{\Gq}{W^q(\bsbeta;\Omega)}
\newcommand{\Gqplus}{W^q_{0,+}(\bsbeta;\Omega)}
\newcommand{\Ginf}{W^\infty(\bsbeta;\Omega)}
\newcommand{\Ginfplus}{W^\infty_{0,+}(\bsbeta;\Omega)}
\newcommand{\Ginfminus}{W^\infty_{0,-}(\bsbeta;\Omega)}

\newcommand{\sign}{\operatorname{sign}}

\newcommand{\Cae}{\widetilde{C}}

\definecolor{darkgray}{gray}{0.4}
\definecolor{ddarkgray}{gray}{0.2}
\definecolor{redgray}{rgb}{0.5,0.25,0.25}
\definecolor{bluegray}{rgb}{0.25,0.25,0.5}

\renewcommand*{\thefootnote}{\fnsymbol{footnote}}
\title{\Large \textsc{%
The Discrete-Dual Minimal-Residual Method (\textsf{DDMRes}) for Weak Advection-Reaction Problems in Banach spaces 
}%
}
\author{\large\textbf{
I.~Muga\footnotemark[2]\, $\boldsymbol{\,\cdot\,}$ 
M.J.W.~Tyler\footnotemark[3]\, $\boldsymbol{\,\cdot\,}$ 
K.G.~van der Zee\footnotemark[4]}
\\
{\footnotesize {1$^{\mathrm{st}}$ August, 2018}}}
\date{\vspace*{-1.5\baselineskip}}
\newlength{\bigfboxsep}
\newcommand{\bigfbox}[1]{\setlength{\fboxsep}{\bigfboxsep}\fbox{#1}}

\newcommand{\negquad}{\mspace{-18.0mu}}
\newcommand{\negqquad}{\negquad\negquad}

\newcommand{\norm}[1]{{\|#1\|}}
\newcommand{\bignorm}[1]{\big\|#1\big\|}

\newcommand{\snorm}[1]{{|#1|}}

\newcommand{\enorm}[1]{{|\! |\! | #1 |\! |\! |}}

\newcommand{\dual}[1]{\langle#1\rangle}
\newcommand{\bigdual}[1]{\big\langle#1\big\rangle}
\newcommand{\Bigdual}[1]{\Big\langle#1\Big\rangle}

\newcommand{\jump}[1]{{[\! [#1 ]\! ]}}

\newcommand{\ds}[1]{{\displaystyle{#1}}}

\renewcommand{\div}{\operatorname{div}}
\providecommand{\grad}{\nabla}
\renewcommand{\grad}{\nabla}

\newcommand{\Span}{\operatorname{Span}}

\newcommand{\pd}{\partial}

\newcommand{\bs}[1]{\boldsymbol{#1}}

\newcommand{\bsn}{\bs{n}}

\newcommand{\bsbeta}{\bs{\beta}}

\newcommand{\mcal}[1]{\mathcal{#1}}

\newcommand{\mcC}{\mcal{C}}

\newcommand{\mcF}{\mcal{F}}

\newcommand{\mcT}{\mcal{T}}

\newcommand{\<}{\left<}
\renewcommand{\>}{\right>}

\begin{document}
%
%
\maketitle
%
%
\footnotetext[2]{Pontificia Universidad Cat\'olica de Valpara\'iso, Instituto de Matem\'aticas,\\ \hspace*{2em}\mbox{ignacio.muga@pucv.cl}}
\footnotetext[3]{University of Nottingham, School of Mathematical Sciences,\\ \hspace*{2em}\mbox{pmymt7@nottingham.ac.uk}}
\footnotetext[4]{University of Nottingham, School of Mathematical Sciences,\\ \hspace*{2em}\mbox{kg.vanderzee@nottingham.ac.uk}}
\renewcommand{\thefootnote}{\arabic{footnote}}
%
%
\begin{abstract}
\noindent
\normalsize
We propose and analyse a minimal-residual method in discrete dual norms for approximating the solution of the advection-reaction equation in a weak Banach-space setting. The weak formulation allows for the direct approximation of solutions in the Lebesgue $L^p$-space, $1<p<\infty$. The greater generality of this weak setting is natural when dealing with rough data and highly irregular solutions, and when enhanced qualitative features of the approximations are needed.
\par
We first present a rigorous analysis of the well-posedness of the underlying continuous weak formulation, under natural assumptions on the advection--reaction coefficients. The main contribution is the study of several discrete subspace pairs guaranteeing the discrete stability of the method and quasi-optimality in~$L^p$, and providing numerical illustrations of these findings, including the elimination of Gibbs phenomena, computation of optimal test spaces, and application to 2-D advection.


\end{abstract}
{\footnotesize
\parbox[t]{\textwidth}{%
\textbf{Keywords~} 
Residual minimization 
$\cdot$~Discrete dual norms
$\cdot$~DDMRes
$\cdot$~Advection--reaction
$\cdot$~Banach spaces 
$\cdot$~Fortin condition
$\cdot$~Compatible FE pairs
$\cdot$~Petrov--Galerkin method
\\[.5\baselineskip]
\textbf{Mathematics Subject Classification~}
41A65 
$\cdot$~65J05 
$\cdot$~46B20 
$\cdot$~65N12 
$\cdot$~65N15 
$\cdot$~35L02 
$\cdot$~35J25 
}
}
\newpage
\setcounter{tocdepth}{2}
{
\tableofcontents
}
%
%
\pagestyle{myheadings}
\thispagestyle{plain}
\markboth{\small \textit{I.~Muga, M.~Tyler and K.G.~van der Zee}}{\small \textit{Weak Advection--Reaction in Banach Spaces}}
%
\newcommand{\DDResMin}{\mbox{\textsf{DDMRes}}}
\section{Introduction}
\label{sec:intro}
Residual minimization encapsulates the idea that an approximation to the solution~$u\in\mbbU$ of an (infinite-dimensional) operator equation $Bu = f$, can be found by minimizing the norm of the residual $f-Bw_n$ amongst all $w_n$ in some finite-dimensional subspace~$\mbbU_n\subset \mbbU$. This powerful idea provides a stable and convergent discretization method under quite general assumptions, i.e., when $B:\mbbU\rightarrow \mbbV^*$ is any linear continuous bijection from Banach space~$\mbbU$~onto the dual~$\mbbV^*$ of a Banach space~$\mbbV$,  $f\in \mbbV^*$, and $\operatorname{dist}(u,\mbbU_n) \rightarrow 0$ as $n\rightarrow \infty$; see, e.g.,~Guermond~\cite[Section~2]{GueSINUM2004} for details. Note that this applies to well-posed weak formulations of linear partial differential equations (PDEs), in which case~$B$ is induced by the underlying bilinear form (i.e., $\dual{Bw,v} = b(w,v) \,, \,\forall w\in \mbbU, \,\forall v\in \mbbV$). As such, residual minimization is essentially an ideal methodology for non-coercive and/or nonsymmetric problems. 
\par
However, for many weak formulations of PDEs, $\mbbV^*$ is a \emph{negative} space (such as $H^{-m}(\Omega)$, or more generally $W^{-m,p}(\Omega)$, which is the dual of the Sobolev space~$W^{m,q}_0(\Omega)$, where $1<p<\infty$, $p^{-1}+q^{-1}=1$, $m=1,2,\ldots,$ or the dual of a~graph space). In that case, this requires the minimization of the residual in the \emph{non-computable} dual norm~$\norm{\cdot}_{\mbbV^*}$. To make this tractable, one can instead minimize in a \emph{discrete} dual norm. In other words, one aims to find an approximation~$u_n\in\mbbU_n$ such that~$\bignorm{f-B u_n}_{(\mbbV_m)^*}$ is minimal, where~$\mbbV_m$ is some finite-dimensional subspace of~$\mbbV$. We refer to this discretization method as residual minimization in discrete dual norms, or simply as the~\DDResMin{} method (\emph{Discrete-Dual Minimal-Residual} method).
\par
%
In this paper, we consider the~\DDResMin{} method when applied to a canonical linear first-order PDE in weak Banach-space settings. In particular, we consider the advection--reaction operator $u\mapsto \bsbeta \cdot \grad u + \mu \,u$, with~$\bsbeta:\Omega\rightarrow \mathbb{R}^d$ and $\mu:\Omega\rightarrow \mathbb{R}$ given advection--reaction coefficients, in a functional setting for which the solution space~$\mbbU$ is~$L^p(\Omega)$, $1<p<\infty$, and $\mbbV$ is a suitable Banach graph space (see Section~\ref{sec:weakAdvRea} for details). 
This weak setting allows for the direct approximation of \emph{irregular} solutions, while the greater generality of Banach spaces (over more common Hilbert spaces) is useful for example in the extension to nonlinear hyperbolic PDEs~\cite{HolRisBOOK2015},%
\footnote{Cf.~\cite{ChaDemMosCF2014, CarBriHelWri2017, CanHeuHAL2018} for nonlinear PDEs examples in Hilbert-space settings using a DPG approach.} 
as well as in approximating solutions with discontinuities (allowing the elimination of Gibbs phenomena; see further details below).
%
%
\par
It has recently become clear that many methods are equivalent to~\DDResMin, the most well-known being the \emph{discontinuous Petrov--Galerkin} (DPG) method (for which $B$~corresponds to a hybrid formulation of the underlying problem, so that $\mbbV$ is a broken Sobolev-type space), see Demkowicz and Gopalakrishnan~\cite{DemGopBOOK-CH2014}, and the \emph{Petrov--Galerkin method with projected optimal test spaces}, see~Dahmen et al.~\cite{DahHuaSchWelSINUM2012}. While these methods require~$\mbbU$ and $\mbbV$ to be Hilbert spaces, in more general Banach spaces the \DDResMin{} method is equivalent to certain (inexact) \emph{nonlinear Petrov--Galerkin} methods, or equivalently, mixed methods with monotone nonlinearity, where the nonlinearity originates from the nonlinear \emph{duality map}~$J_\mbbV:\mbbV\rightarrow \mbbV^*$; see Muga~\& Van~der~Zee~\cite{MugZeeARXIV2018} for details, including a schematic overview of connections to other methods. 
\par
%
%
%
%
The numerical analysis of the \DDResMin{} method has been carried out abstractly by Gopalakrishnan~\& Qiu~\cite{GopQiuMOC2014} in Hilbert spaces (see also~\cite[Section~3]{DahHuaSchWelSINUM2012}), and by Muga~\& Van~der~Zee~\cite{MugZeeARXIV2018} in smooth Banach spaces. A key requirement in these analyses is the \emph{Fortin} compatibility condition on the family of discrete subspace pairs~$(\mbbU_n,\mbbV_m)$ under consideration, which, once established, implies stability and quasi-optimal convergence of the method. In some sense, the Fortin condition is rather mild, since for a given~$\mbbU_n$, there is the expectation that it will be satisfied for a sufficiently large~$\mbbV_m$ (thereby making the discrete dual norm~$\norm{\cdot}_{(\mbbV_m)^*}$ sufficiently close to~$\norm{\cdot}_{\mbbV^*}$). Of course, whether this can be established depends crucially on the operator~$B$, therefore also on the particular weak formulation of the PDE that is being studied. 
%
\par
The main contribution of this paper consists in the study of several elementary discrete subspace pairs~$(\mbbU_n,\mbbV_m)$ for the \DDResMin{} method for weak advection--reaction, including proofs of Fortin compatibility in the above-mentioned Banach-space setting. It thereby provides the first application and corresponding analysis of~\DDResMin{} in genuine (non-Hilbert) Banach spaces. In particular, for the given compatible pairs, \DDResMin{} is thus a quasi-optimal method providing a near-best approximation in~$L^p(\Omega)$. 
%
%
Note that our results do not cover DPG-type hybrid weak formulations (with a broken graph space~$\mbbV$), so that our discrete spaces~$\mbbV_m$ are globally conforming. Broken Banach-space settings will be treated in forthcoming work. 
\par
We now briefly discuss some details of our results. To be able to carry out the analysis, our results focus on discrete subspace pairs~$(\mbbU_n,\mbbV_m)$, where $\mbbU_n$ is a lowest-order finite element space on mesh~$\mcT_n$ in certain specialized settings. 
We first consider \emph{continuous linear} finite elements in combination with continuous finite elements of degree~$k$, i.e., $\mbbU_n = \mbbP^1_{\mathrm{cont}}(\mcT_n)$ and~$\mbbV_m = \mbbP^k_{\mathrm{cont}}(\mcT_n)$. The Fortin condition holds when $k\ge 2$, assuming, e.g.,~incompressible pure advection ($\div \bsbeta =  \mu = 0$) in a one-dimensional setting. Interestingly, we demonstrate that the notorious \emph{Gibbs phenomenon} of spurious numerical over- and undershoots, commonly encountered while approximating discontinuous solutions with continuous approximations, can be \emph{eliminated} with the \DDResMin{}~method upon $p\rightarrow 1^+$  (see Section~\ref{sec:Gibbs}), which is in agreement with previous findings on $L^1$-methods~\cite{GueSINUM2004, LavSINUM1989}.
\par
We then consider $\mbbU_n = \mbbP^0(\mcT_n)$, that is, \emph{discontinuous piecewise-constant} approximations on arbitrary partitionings of the domain~$\Omega$ in~$\mathbb{R}^d$, $d\ge 1$. 
It turns out that it is possible to define an \emph{optimal} test space~$\mbbS_n := B^{-*}\mbbU_n$ and subsequently prove Fortin's condition for any~$\mbbV_m \supseteq \mbbS_n$. This result essentially hinges on the fact that $\mbbU_n$ is invariant under the $L^p$~duality map (see proof of Proposition~\ref{prop:AdvReac_compatible}). 
Since the optimal test space is however not explicit, it requires in general the computation of an explicit basis (see Section~\ref{sec:optimalTestSpace}). Such computations may not be feasible in practise, and in those cases, as an alternative, one could resort to sufficiently-rich~$\mbbV_m$, e.g.,~continuous linear finite elements on a sufficiently-refined submesh of the original mesh (cf.~\cite{BroDahSteMOC2018}). 
\par
Interestingly however, under certain special, yet nontrivial situations, the optimal test space~$\mbbS_n$ happens to coincide with a convenient finite element space. For example, in 2-D in the incompressible pure advection case with $\bsbeta$ piecewise constant on some partition, if $\mcT_n$ is a triangular mesh of~$\Omega$ (compatible with the partition) and all triangles are flow-aligned, then we prove that $\mbbS_n = \mbbP^1_{\mathrm{conf}}(\mcT_n)$, where $\mbbP^1_{\mathrm{conf}}(\mcT_n)$ refers to the space of piecewise-linear functions that are \emph{conforming} with respect to the graph space~$\mbbV$. Numerical experiments in 2-D indeed confirm in this case the quasi-optimality of~\DDResMin{} (see Section~\ref{sec:2DflowAlignedMesh}).
%
%
%
\par
In recent years, several similar methods for weak advection--reaction have appeared, all of which were in Hilbert-space settings (i.e., the solution space is~$L^2(\Omega)$) and use a broken weak formulation. Indeed, these include some of the initial DPG methods~\cite{DemGopCMAME2010, DemGopNMPDE2011, BuiDemGhaMOC2013}, which were proposed before the importance of Fortin's condition was clarified. Recently however, Broersen, Dahmen \&~Stevenson~\cite{BroDahSteMOC2018} studied a higher-order pair using standard finite-element spaces for the DPG method of weak advection--reaction. Under mild conditions on~$\bsbeta$, they proved Fortin's condition when~$\mbbU_n$ consists of piecewise polynomials of degree~$k$, and~$\mbbV_m$ consists of piecewise polynomials of higher degree over a sufficiently-deep refinement of the trial mesh. The extension of their proof, based on approximating the optimal test space, to any Banach-space setting seems nontrivial, since currently the concept of an optimal test space is in general absent in \DDResMin{} in Banach spaces (cf.~\cite{MugZeeARXIV2018}), exceptions notwithstanding (such as the lowest-order piecewise-constant case discussed above).  
\par
Let us finally point out that methods for weak advection--reaction are quite distinct from methods for \emph{strong} advection--reaction (which has its residual in~$L^p(\Omega)$ and a~priori demands more regularity on its solution).  Indeed, there is a plethora of methods in the strong case; see, e.g., Ern \&~Guermond~\cite[Chapter~5]{ErnGueBOOK2004} and Guermond~\cite{GueSINUM2004}, all of which typically exhibit \emph{suboptimal} convergence behaviour when measured in~$L^p(\Omega)$. In the context of \emph{strong} advection--reaction the results by Guermond~\cite{GueM2AN1999} are noteworthy, who proved the Fortin condition for several pairs, consisting of a low-order finite element space and its enrichment with bubbles. These results however do not apply to weak advection--reaction. Similarly for the stability result by Chan, Evans \&~Qiu~\cite{ChaEvaQiuCAMWA2014}.
\par
The remainder of this paper is arranged as follows. 
In Section~\ref{sec:prelim}, we first present preliminaries for the advection--reaction equation, allowing us to recall in Section~\ref{sec:weakAdvRea} the specifics of the well-posed Banach-space setting (cf.~Cantin~\cite{CanCR2017}). In particular, we provide a self-contained proof of the continuous inf-sup conditions using various properties of the $L^p$~duality map. Then, in Section~\ref{sec:AdvRea_discrete}, we consider the discrete problem corresponding to the \DDResMin{}~method in the equivalent form of the monotone mixed method, and establish stability and quasi-optimality of the method, provided the Fortin condition holds. In Section~\ref{sec:applications}, we consider particular discrete subspace pairs~$(\mbbU_n,\mbbV_m)$. This section contains several proofs of Fortin conditions, as well as some illustrative numerical examples pertaining to the Gibbs phenomena (Section~\ref{sec:Gibbs}), optimal test space basis (Section~\ref{sec:optimalTestSpace}), and quasi-optimal convergence for 2-D advection (Section~\ref{sec:2DflowAlignedMesh}).
\section{Advection--reaction preliminaries}
\label{sec:prelim}
%
For any dimension $d\geq1$, let $\Omega\subset\mathbb R^d$ be an open bounded domain, with Lipschitz boundary $\partial\Omega$ oriented by a unit outward normal vector $\bsn$. 
Let $\bsbeta\in L^\infty(\Omega)$ be an advection-field such that $\div \bsbeta\in L^\infty(\Omega)$ and let $\mu\in L^\infty(\Omega)$ be a (space-dependent) reaction coefficient. The advection-field splits the boundary~$\partial\Omega$ into an \emph{inflow}, \emph{outflow} and \emph{characteristic} part, which for continuous~$\bsbeta$ corresponds to
\begin{alignat*}{4}
& \partial\Omega_- 
&&:=\big\{x\in\partial\Omega:& \; \bsbeta(x)\cdot\bsn(x) & <0 & \big\}\,, 
\\
& \partial\Omega_+ 
&&:=\big\{x\in\partial\Omega:& \; \bsbeta(x)\cdot\bsn(x) & >0 & \big\}\,,
\\
& \partial\Omega_0 
&&:=\big\{x\in\partial\Omega:& \; \bsbeta(x)\cdot\bsn(x) & =0 & \big\}\,,
\end{alignat*}
respectively; see~\cite[Section~2]{BroDahSteMOC2018} for the definition of the parts in the more general case~$\bsbeta, \div \bsbeta \in L^\infty(\Omega)$ (which is based on the integration-by-parts formula~\eqref{eq:IntByP}). 
\par
Given a~possibly \emph{rough} source~$f_\circ$ and inflow data~$g$, the advection--reaction model is:
\begin{subequations}
\label{eq:strong_advreact}
\begin{empheq}[left=\left\{\,,right=\right.,box=]{alignat=2}
\bsbeta\cdot \nabla u + \mu u &= f_\circ  &\qquad & \text{in } \Omega\,,
\\
 u &= g && \text{on } \partial\Omega_-.
\end{empheq}
\end{subequations}
Before we give a weak formulation for this model and discuss it's well-posedness, we first introduce relevant assumptions and function spaces. We have in mind a weak setting where~$u\in L^p(\Omega)$, for any~$p$ in~$(1,\infty)$. Therefore, throughout this section, let~$1<p<\infty$ and let~$q\in (1,\infty)$ denote the \emph{conjugate} exponent of~$p$, satisfying the relation ${p}^{-1}+{q}^{-1}=1$.
\par
The following assumptions are natural extensions of the classical ones in the Hilbert case.
%
\begin{assumption}[Friedrich's positivity assumption]
\label{assump:mu_0}
There exists a constant~$\mu_0>0$ for which:
\begin{equation}\label{eq:beta-mu}
\mu(x) -{1\over p}\div\bsbeta(x)\geq \mu_0, \quad\hbox{ a.e. $x$ in } \Omega.
\end{equation}
\end{assumption}
%
\begin{assumption}[Well-separated in- and outflow]
\label{assump:split}
The in- and outflow boundaries are \emph{well-separated}, i.e., $\overline{\partial\Omega_-}\cap\overline{\partial\Omega_+}=\emptyset$ and, by partition of unity, there exists a function 
\begin{equation}\label{eq:cutoff}
\left\{
\begin{array}{l}
\phi\in C^\infty(\overline\Omega) \hbox{ such that:}\\
\quad \phi(x) =1, \quad\forall x\in \partial\Omega_-,\\
\quad \phi(x) =0, \quad\forall x\in \partial\Omega_+.
\end{array}\right.
\end{equation}
\end{assumption}
%
\par
For brevity we use the following notation for norms and duality pairings: 
\begin{subequations}
\label{eq:s-norm}
\begin{alignat}{4}
& \|\cdot\|_\infty &&=  \operatorname*{ess~sup}_{x\in \Omega}|\cdot(x)|\,,
\\
& \|\cdot\|_\rho &&=\bigg(\int_\Omega |\cdot|^\rho\bigg)^{1/\rho}\,,
&\qquad & \text{for } 1\le \rho < \infty\,,
\\
& \dual{\cdot,\cdot}_{\rho,\sigma} &&= \bigdual{\cdot,\cdot}_{L^\rho(\Omega),L^\sigma(\Omega)}\,,
&\qquad & \text{for } 1< \rho < \infty\,, \quad \sigma = \frac{\rho}{\rho-1}\,.
\end{alignat}
\end{subequations}
%
\begin{definition}[Graph space]
For $1\le \rho \le \infty$, the \emph{graph space} is defined by
$$
 \Grho :=\Big\{w\in L^\rho(\Omega) : \bsbeta\cdot\nabla w\in L^\rho(\Omega) \Big\},
$$
endowed with the norm 
\begin{alignat*}{2}
\|w\|^2_{\rho,\bsbeta}:=
\|w\|^2_\rho+\|\bsbeta\cdot\nabla w\|^2_\rho\,.
\end{alignat*}
The ``adjoint'' norm is defined by 
%
\begin{alignat}{2}
\label{eq:beta-norm}
\enorm{w}^2_{\rho,\bsbeta}:=\|w\|^2_\rho+\|\div(\bsbeta w) \|^2_\rho\,\,.
\end{alignat}
These norms are equivalent, which can be shown by means of the identity 
\begin{alignat}{2}\label{eq:beta_identity}
\div(\bsbeta w)=\div(\bsbeta)w + \bsbeta\cdot\nabla w\,.
\end{alignat}
\end{definition}
%
\begin{remark}[Graph-spaces and traces]
\label{rem:traces}
As a consequence of Assumption~\ref{assump:split}, traces on the space~$\Grho$ are well-defined as functions in the space
\begin{alignat*}{2}
L^\rho(|\bsbeta\cdot\bsn|;\partial\Omega):=\left\{w \hbox{ measurable in } \partial\Omega: \int_{\partial\Omega}|\bsbeta\cdot\bsn|\,|v|^\rho < +\infty\right\},
\end{alignat*}
and, moreover, for all $w\in \Grho$ and all $v\in \Gsigma$, the following integration-by-parts formula holds:
\begin{alignat}{2}
\label{eq:IntByP}
\int_\Omega \Big( (\bsbeta\cdot\nabla w)v+(\bsbeta\cdot\nabla v)w
+\div(\bsbeta)wv \Big)  = \int_{\partial\Omega}
(\bsbeta\cdot\bsn)wv\,.
\end{alignat}
The proof of these results is a straightforward extension of the Hilbert-space case given in, e.g., Di~Pietro~\& Ern~\cite[Section~2.1.5]{DipErnBOOK2012} (cf.~Dautray~\& Lions~\cite[Chapter~XXI, \S{}2, Section~2.2]{DauLioBOOK1993}~and~Cantin~\cite[Lemma~2.2]{CanCR2017}). We can thus define the following two closed subspaces, which are relevant for prescribing boundary conditions at $\pd\Omega_+$ or $\pd\Omega_-$:
\begin{alignat}{2}\label{eq:Wbeta}
\Grhopm
:=\Big\{w\in \Grho : w\big|_{\partial\Omega_\pm}=0 \Big\}\,.
\end{alignat}
%
%
%
\end{remark}
%
\begin{remark}[Non-separated in- and outflow]
The requirement of separated in- and outflow can be removed, but different trace operators have to be introduced~\cite{GopMonSepCAMWA2015}.
\end{remark}
%
%
\par
The case when $\mu\equiv 0$ and $\div\bsbeta\equiv 0$ is special, 
since Assumption~\ref{assump:mu_0} is not satisfied. An important tool for the analysis of this case is the so-called \emph{curved Poincar\'e-Friedrichs inequality}; see Lemma~\ref{lem:Poincare_Friedrichs} below. Its proof relies on the following assumption (cf.~\cite{AzePhD1996, AzePouCR1996}).
%
\begin{assumption}[$\Omega$-filling advection]\label{ass:omega-filling}
Let~$1<\rho<\infty$. If $\mu\equiv 0$ and $\div\bsbeta\equiv 0$, the 
advection-field~$\bsbeta$ is \emph{$\Omega$-filling}, by which we mean that there exist
$z_+,z_- \in \Ginf$ with~$\norm{z_+}_\infty, \norm{z_-}_\infty>0$, 
such that 
\begin{equation}\left\{
\begin{array}{rl}
\label{eq:omega-filling}
-\bsbeta\cdot\nabla z_\pm = \rho & \quad \text{in } \Omega\,, 
\\
 z_{\pm}  = 0 &\quad  \text{on } \pd\Omega_{\pm}\,.
\end{array}\right.
\end{equation}
%
%
%
\end{assumption}
\par
\begin{remark}[Method of characteristics]
Assumption~\ref{ass:omega-filling} holds, for example, if $\bsbeta$ is regular enough so that the method of characteristics can be employed to solve for~$z$ (cf.~Dahmen et al.~\cite[Remark 2.2]{DahHuaSchWelSINUM2012}).
%
%
\end{remark}
%
\begin{lemma}[Curved Poincar\'e--Friedrichs inequality]
\label{lem:Poincare_Friedrichs}
Let~$1<\rho<\infty$. Under the hypothesis that Assumption~\ref{ass:omega-filling} holds true, there exists a constant~$\CPF>0$ such that 
\begin{alignat}{2}
\|w\|_\rho \leq \CPF \|\bsbeta\cdot\nabla w\|_\rho\,\,, 
\quad \forall w\in \Grhopm\,.
\end{alignat}
\end{lemma}
%
%
\begin{proof}
For the Hilbert-space case ($\rho=2$), the proof can be found in~\cite{AzePouCR1996}. For completeness, we reproduce here the general $\rho$-version.
\par
Without loss of generality take $w\in \Grhominus$, and let~$z=z_+\in \Ginfplus$ as in Assumption~\ref{ass:omega-filling}. 
Notice the important identity
\begin{alignat}{2}
\label{eq:rho_id}
 z \,\bsbeta\cdot\nabla(|w|^\rho)=\div(\bsbeta  z |w|^\rho)-|w|^\rho\bsbeta\cdot\nabla z .
\end{alignat}
Let $\sigma={\rho/(\rho-1)}$. 
Take $\phi_w= \rho z |w|^{\rho-1}\sign(w)\in L^{\sigma}(\Omega)$, which satisfies
\begin{alignat}{2}\label{eq:phi_u}
\|\phi_w\|_{\sigma}\leq   \rho \,\| z \|_\infty
\|w\|_\rho^{\rho-1}\,\,.
\end{alignat}
Thus,
\begin{alignat}{2}
\tag{by duality}
\|\bsbeta\cdot\nabla w\|_\rho  & = \sup_{0\neq \phi\in L^{\sigma}(\Omega)}
{\bigdual{\bsbeta\cdot\nabla w,\phi}_{\rho,\sigma}\over \|\phi\|_{\sigma}}\\
\tag{since $\phi_w\in L^{\sigma}(\Omega)$}
 & \geq
{\bigdual{\bsbeta\cdot\nabla w,\phi_w}_{\rho,\sigma}\over \|\phi_w\|_{\sigma}}\\
\tag{by \eqref{eq:rho_id} and~$\int_{\pd\Omega} (\bsbeta\cdot\bsn) z |w|^\rho = 0$}
& = 
{\displaystyle
-\int_\Omega|w|^\rho\bsbeta\cdot\nabla z \over \|\phi_w\|_{\sigma}}\\
\tag{by Assumption~\ref{ass:omega-filling} and~\eqref{eq:phi_u}}
& \geq {\|w\|_\rho\over \| z \|_\infty}.
\end{alignat}
Hence, $\CPF=\| z \|_\infty$. If $w\in \Grhoplus$, take~$z=z_{-}\in \Ginfminus$.
\end{proof}
The proof of Lemma~\ref{lem:Poincare_Friedrichs} shows that~$\CPF=\norm{z_\pm}_\infty$, with~$z_\pm$ defined in~\eqref{eq:omega-filling}, hence $\CPF$ depends on~$\Omega$, $\bsbeta$ and~$\rho$.
%
\begin{remark}[Weaker statements]
Under a weaker condition than Assumption~\ref{assump:mu_0}, Lemma~\ref{lem:Poincare_Friedrichs} can be generalized to the following situations:
\begin{alignat}{2}
\|w\|_p & \lesssim  \|\mu w + \bsbeta\cdot\nabla w\|_p\,, \quad \forall w\in \Gpminus\,,\\
\|v\|_q & \lesssim  \|\mu v - \div(\bsbeta v)\|_q\,,  \quad \forall v\in \Gqplus\,.
\end{alignat}
Indeed, it is enough to verify the existence of a constant $\mu_0^*>0$ and a Lipschitz continuous function $\zeta(x)$ such that:
\begin{alignat}{2}\label{eq:weak beta-mu}
\mu(x) -{1\over p}\div\bsbeta(x) - {1\over p}\bsbeta(x)\cdot\nabla\zeta(x)
\geq \mu_0^*\,, \quad\hbox{ a.e. $x$ in } \Omega.
\end{alignat}
These statements can be inferred from the recent work of Cantin~\cite{CanCR2017}.  Notice that if Assumption~\ref{assump:mu_0} is satisfied, then~\eqref{eq:weak beta-mu} holds with $\zeta(x)\equiv 0$ and $\mu_0^*=\mu_0$.
\end{remark}
%
\section{A weak setting for advection--reaction}\label{sec:weakAdvRea}
The weak setting for the advection-reaction problem~\eqref{eq:strong_advreact} considers a trial space $\mathbb U:=L^p(\Omega)$ endowed with the $\|\cdot\|_p$-norm (see~\eqref{eq:s-norm}), and a test space $\mathbb V:=\Gqplus$ endowed with the norm $\enorm{\cdot}_{q,\bsbeta}$ (see~\eqref{eq:beta-norm}). The weak-formulation reads as follows:
\begin{empheq}[left=\left\{\,,right=\right.,box=]{alignat=2}
\notag & \text{Find }  u\in \mathbb U=L^p(\Omega) : 
\\
\label{eq:weak_advection-reaction}
& \quad  \dual{Bu,v}_{\mbbV^*,\mbbV}
 = \dual{f,v}_{\mathbb V^*,\mathbb V}\,\,,\quad \forall v\in \mathbb V=\Gqplus.
\end{empheq}
where $B:\mbbU\to\mbbV^*$ is defined by
\begin{equation}\label{eq:B_AdvReac}
\dual{Bw,v}_{\mbbV^*,\mbbV}:=\int_\Omega w\big (\mu v - \div(\bsbeta v) \big)\,,\qquad\forall w\in \mbbU, \,\forall v\in\mbbV,
\end{equation}
and 
the right-hand side~$f$ is related to the original PDE data~$(f_\circ,g)$ via:
\begin{alignat*}{2}
 \dual{f,v}_{\mathbb V^*,\mathbb V}
  = \int_\Omega f_\circ v  + \int_{\pd\Omega_-} |\bsbeta\cdot\bsn| g v \,,
\end{alignat*}
where $f_\circ$ is given in (for example) $L^p(\Omega)$ and $g$ is given in $L^p(|\bsbeta\cdot\bsn|;\pd\Omega)$. More rough~$f_\circ$ is allowed as long as $f\in [\Gqplus]^*$.
\begin{remark}[Boundedness]
\label{rem:mu_continuity}
The bilinear form in~\eqref{eq:weak_advection-reaction} is bounded with constant 
$M_\mu:=\sqrt{1+\|\mu\|_\infty^2}$. Indeed,
\begin{alignat*}{2}
\left|\displaystyle\int_\Omega u\left(\mu v-\div(\bsbeta v)\right)\right|\leq \|u\|_p \bignorm{\mu v-\div(\bsbeta v)}_q
\leq M_\mu\,\|u\|_p\,\enorm{v}_{q,\bsbeta}\,.
\end{alignat*}
\end{remark}
\par
The following result states the well-posedness of problem~\eqref{eq:weak_advection-reaction}. Although this result can be inferred from the recent result by Cantin~\cite{CanCR2017}, we provide a slightly alternative proof based on establishing the adjoint inf-sup conditions using properties of the $L^p$~duality map. For a classical proof of well-posedness in a similar Banach-space setting, we refer to~Beir\~{a}o Da Veiga~\cite{BeiRDM1987, BeiRSMUP1988} (cf.~\cite{BarLerNedCPDE1979} and~\cite[Chapter~XXI]{DauLioBOOK1993}). 
%
\begin{theorem}[Weak advection--reaction: Well-posedness]
\label{thm:avdreact_wellposed} 
Let~$1<p<\infty$ and $p^{-1} + q^{-1} = 1$. Let $\Omega\subset \mbbR^d$ be an open bounded domain with Lipschitz boundary. Let $\bsbeta:\Omega\to \mbbR$ and $\mu:\Omega\to \mbbR$ be advection and reaction coefficients (respectively) satisfying either Friedrich's positivity Assumption~\ref{assump:mu_0}, or the $\Omega$-filling Assumption~\ref{ass:omega-filling}. Assume further that in- and outflow boundary are well~separated (Assumption~\ref{assump:split}).
\begin{enumerate}[(i)]
\item
For any $f\in \mathbb V^*=[\Gqplus]^*$, there exists a unique solution $u\in L^p(\Omega)$ to the weak advection--reaction problem~\eqref{eq:weak_advection-reaction}.
\item
In the case that Assumption~\ref{assump:mu_0} holds true, we have the following a priori bound:
\begin{equation}\label{eq:C_mu}
\|u\|_p\leq \frac{1}{\gamma_{B}}\,\|f\|_{\mathbb V^*}\,, \qquad\hbox{ with} \quad \gamma_B = \sqrt{\mu_0^2\over
{1+(\mu_0+\|\mu\|_\infty)^2}}
\end{equation}
and $\mu_0>0$ being the constant in Assumption~\ref{assump:mu_0}. 
\item[(ii$\star$)]
On the other hand, in the case where Assumption~\ref{ass:omega-filling} holds true, we also have the a~priori bound~\eqref{eq:C_mu}, but $\gamma_B$ in~\eqref{eq:C_mu} must be replaced by the constant $1/(1+\CPF)$, where $\CPF>0$ is the Poincar\'e--Friedrichs constant in Lemma~\ref{lem:Poincare_Friedrichs}.
\end{enumerate}
~
\end{theorem}
\begin{proof}
See Section~\ref{sec:avdreact_wellposed}.
\end{proof}
%
\section{The general discrete problem}
\label{sec:AdvRea_discrete}
We now consider the approximate solution of~\eqref{eq:weak_advection-reaction} given by the \DDResMin{} method, i.e., given finite-dimensional subspaces $\mathbb U_n\subset \mathbb U=L^p(\Omega)$ and 
$\mbbV_m\subset\mbbV=\Gqplus$, we aim to find $u_n\in\mathbb U_n$ such that:
\begin{equation}
\label{eq:discrete_residual}
u_n=\arg\min_{w_n\in \mbbU_n}\bignorm{f-Bw_n}_{(\mbbV_m)^*}\,\,,
\end{equation}
where the discrete dual norm is given by
\begin{alignat*}{2}
 \norm{\cdot}_{(\mbbV_m)^*} =
 \sup_{v_m\in \mbbV_m}{\< \,\cdot\, ,v_m\>_{(\mbbV_m)^*,\mbbV_m}\over \|v_m\|_{\mbbV}} 
\,.
\end{alignat*}
As proven in~\cite[Theorem~4.A]{MugZeeARXIV2018}, the minimization problem~\eqref{eq:discrete_residual} is equivalent to following monotone mixed method:
\begin{subequations}
\label{eq:AdvReac_discrete}
\begin{empheq}[left=\left\{\,,right=\right.,box=]{alignat=3}\notag
& \text{Find }  (r_m,u_n)\in \mathbb V_m\times\mathbb U_n \text{ such that} \negqquad
\\ \label{eq:AdvReac_discrete_a}
& \quad \bigdual{J_\mathbb V(r_m),v_m}_{\mathbb V^*,\mathbb V}
  +\<Bu_n,v_m\>_{\mathbb V^*,\mathbb V}
 &&  =\dual{f,v_m}_{\mathbb V^*,\mathbb V}\,, 
  &\quad & \forall v_m\in \mathbb V_m\,,
\\\label{eq:AdvReac_discrete_b}
& \quad \<Bw_n,r_m\>_{\mathbb V^*,\mathbb V}
 &&  = 0\,, 
  &\quad & \forall w_n\in \mathbb U_n\,.
\end{empheq}
\end{subequations}
In~\eqref{eq:AdvReac_discrete_a}, 
$J_\mathbb V:\mbbV\to\mbbV^*$ denotes the (monotone and nonlinear) duality map of $\mathbb V=\Gqplus$ defined by the action:
\begin{alignat}{2}
\label{adv:J}
\bigdual{J_\mathbb V(r),v}_{\mathbb V^*,\mathbb V}
:=\bigdual{J_q(r),v}_{p,q}
+\bigdual{J_q\big(\div(\bsbeta r)\big),\div(\bsbeta v)}_{p,q}
\qquad \forall r,v\in \mbbV\,,
\end{alignat}
where $J_q(v):= \|v\|_q^{2-q}|v|^{q-1}\sign v \in L^p(\Omega)$ is the duality map of~$L^q(\Omega)$. The solution $u_n\in\mbbU_n$ of~\eqref{eq:AdvReac_discrete} is exactly the residual minimizer of~\eqref{eq:discrete_residual}, while $r_m\in\mbbV_m$ is a \emph{representative of the discrete residual}, i.e., $J_\mbbV(r_m)=f-Bu_n$ in $(\mbbV_m)^*$.

%
The well-posedness of the discrete method~\eqref{eq:AdvReac_discrete} relies on the well-posedness of the continuous problem~\eqref{eq:weak_advection-reaction} (see Theorem~\ref{thm:avdreact_wellposed}), together with the following \emph{Fortin} assumption. 
\begin{assumption}[Fortin condition]
\label{assumpt:Fortin}
Let $B:\mbbU\to\mbbV^*$ be a bounded linear operator and let $\{(\mbbU_n,\mbbV_m)\}$ be a family of \emph{discrete} subspace pairs, where $\mbbU_n\subset\mbbU$ and $\mbbV_m\subset\mbbV$. For each pair $(\mbbU_n,\mbbV_m)$ in this family, there exists an operator $\Pi_{n,m}:\mbbV\to \mbbV_m$ and constants $C_\Pi>0$ (independent of $n$ and $m$) such that the following conditions are satisfied:
\begin{subequations}
\label{eq:Fortin}
\begin{empheq}[left=\left\{,right=\right.,box=]{alignat=2}
\label{eq:Fortin_a}
\,\, &
 \|\Pi_{n,m} v\|_\mathbb V\leq C_\Pi \|v\|_\mathbb V\,\,, & \quad & \forall v\in\mathbb V\,, 
\\ \label{eq:Fortin_c}
& \<Bw_n,v-\Pi_{n,m} v\>_{\mbbV^\ast,\mbbV}=0, &&  \forall w_n\in \mbbU_n,\,\forall v\in \mbbV,
\end{empheq}
\end{subequations}
where $I:\mathbb V\to\mbbV$ is the identity map in~$\mbbV$. For simplicity, we write $\Pi$ instead of~$\Pi_{n,m}$.
\end{assumption}
\begin{theorem}[Weak advection--reaction: \DDResMin{} method]\label{thm:AdvReac_discrete}
Under the conditions of Theorem~\ref{thm:avdreact_wellposed}, let the pair $(\mbbU_n,\mbbV_m)$ satisfy the (Fortin) Assumption~\ref{assumpt:Fortin} with operator~$B$ given by~\eqref{eq:B_AdvReac}.
\begin{enumerate}[(i)]
\item There exists a unique solution $(r_m,u_n)$ to~\eqref{eq:AdvReac_discrete}, which satisfies the a~priori bounds:
\begin{equation}\label{eq:AdvReac_apriori_bounds}
\enorm{r_m}_{q,\bsbeta}\leq \|f\|_{\mbbV^*}\qquad\hbox{ and }\qquad
\|u_n\|_p \leq \Cae\, \|f\|_{\mbbV^*}\,,
\end{equation}
with $\Cae := C_\Pi\,\left(1+\CAO(\mbbV)\right)/ \gamma_B\,$.
\item Moreover, we have the a~priori error estimate:
\begin{alignat}{2}\label{eq:apriori_AdvReac}
\|u-u_n\|_p & \leq C \inf_{w_n\in\mathbb U_n}\|u-w_n\|_p\,,
\end{alignat}
where 
$\ds{C=\min\Big\{ 2^{\left|{2\over p}-1\right|} M_\mu \widetilde{C} \, , \, 1+M_\mu\widetilde{C} \Big\}}\,$.
\end{enumerate}
The constants involved are: $C_\Pi$ which is given in~Assumption~\ref{assumpt:Fortin}, the boundedness constant $M_\mu$ given in Remark~\ref{rem:mu_continuity}, the stability constant $\gamma_B$ given in~\eqref{eq:C_mu} (see also the statement (ii$\star$) in Theorem~\ref{thm:avdreact_wellposed}), and the geometrical constant $\CAO(\mbbV)$ (for $\mbbV=\Gqplus $) defined in~\cite[Definition~2.14]{MugZeeARXIV2018}.
\end{theorem}
%
\begin{proof}
Statement~(i) directly follows from~\cite[Theorem~4.B]{MugZeeARXIV2018} applied to the current situation, while statement~(ii) follows from~\cite[Theorem~4.D]{MugZeeARXIV2018}, which can be applied since the spaces $\mathbb U=L^p(\Omega)$ and $\mbbV=\Gqplus$ are strictly convex and reflexive for $1<p,q<+\infty$, as well as the dual spaces $\mathbb U^*$ and $\mathbb V^*$. The factor $2^{\left|{2\over p}-1\right|}$ is the value of the Banach--Mazur constant~$\CBM(\mbbU)$ (appearing in~\cite[Theorem~4.D]{MugZeeARXIV2018}) for $\mbbU = L^p(\Omega)$; see~\cite[Section~5]{SteNM2015}.
%
%
\end{proof}
%
%
\begin{remark}[Finite elements]
\label{rem:advFEM}
Theorem~\ref{thm:AdvReac_discrete} implies \emph{optimal convergence rates} in~$L^p(\Omega)$ for finite element subspaces $\mbbU_n \equiv \mbbU_h$, 
provided $C_\Pi$ is uniformly bounded.  
For example, on a sequence of approximation subspaces $\{\mbbP^k(\mcT_h)\}_{h>0}$ of piecewise polynomials of fixed degree~$k$ on quasi-uniform shape-regular meshes~$\mcT_h$ with mesh-size parameter~$h$, well-known best-approximation estimates 
(see, e.g.,~\cite[Section~4.4]{BreScoBOOK2008}, \cite[Section~1.5]{ErnGueBOOK2004} and~\cite{ErnGueM2AN2017}) imply
\begin{alignat*}{2}
 \norm{u-u_n}_p \lesssim  
 \inf_{w_h\in\mathbb U_h}\|u-w_h\|_p\lesssim h^s \snorm{u}_{W^{s,p}(\Omega)}\,,
 \qquad \text{for } 0\leq s\leq k+1\,,
\end{alignat*}
where $\snorm{\cdot}_{W^{s,p}(\Omega)}$ denotes a standard semi-norm of~$W^{s,p}(\Omega)$ (e.g., of Sobolev--Slobodeckij type). For a relevant regularity result in~$W^{1,\rho}(\Omega)$, with~$\rho \ge 2$, see Girault \&~Tartar~\cite{GirTarCR2010} (see also~\cite{PiaARXIV2016}).
%
%
%
\end{remark}
%
%
\section{Applications}
\label{sec:applications}
In this section we apply the general discrete method~\eqref{eq:AdvReac_discrete} to particular choices of discrete subspace pairs~$(\mbbU_n, \mbbV_m)$ involving low-order finite-element spaces. 
%
\par
For simplicity, throughout this section $\Omega \subset \mathbb{R}^d$ will be a \emph{polyhedral} domain and $\mcT_n$ will denote a finite partition of~$\Omega$, i.e., $\mcT_n = \{T\}$ consists of a finite number of non-overlapping elements~$T$ for which~$\overline{\Omega} = \bigcup_{T\in \mcT_n} \overline{T}$. 
%
\subsection{The pair $\mbbP^1_\mathrm{cont}(\mcT_n)$~-~$\mbbP^k_\mathrm{cont}(\mcT_n)$: Eliminating the Gibbs phenomena}
\label{sec:Gibbs}
By first considering \emph{continuous} finite elements for~$\mbbU_n$, we briefly illustrate how the discrete method~\eqref{eq:AdvReac_discrete} eliminates the well-known \emph{Gibbs phenomena} when approaching discontinuous solutions. 
%
%
For simplicity, consider the advection--reaction problem~\eqref{eq:strong_advreact} with $\Omega \equiv (-1,1)\subset \mbbR$, $\beta = 1$, $\mu = 0$, $g=-1$ and let the source~$f_\circ$ be $2 \delta_0$, where $\delta_0$ is the \emph{Dirac delta} at~$x=0$, i.e., 
\begin{subequations}
\label{eq:heavi}
\begin{empheq}[left=\left\{\,,right=\right.,box=]{alignat=2}
   u'(x)  &= 2 \delta_0(x)\,, \qquad \forall x\in (-1,1) \,,
\\
  u(-1) &=-1\,.
\end{empheq}
\end{subequations}
Notice that the exact solution of~\eqref{eq:heavi} corresponds to the sign of~$x$:
\begin{alignat*}{2}
  u(x) & = \sign(x) :=  \left\{
  \begin{array}{ll}
  -1 & \hbox{if } x<0,\\
  \phantom{-} 1 & \hbox{if } x>0.
  \end{array}
  \right.
\end{alignat*}
%
%
\par
We endow $\mbbV = \Gqplus = W^{1,q}_{0,\{1\}}(\Omega)$ with the norm $\|\cdot\|_\mbbV=\|(\cdot)'\|_q$, which simplifies the duality map~\eqref{adv:J} to a normalized \emph{$q$-Laplace operator}:
\begin{alignat}{2}
\label{adv:J1D} 
  \bigdual{J_\mbbV(r),v}_{\mbbV^*,\mbbV} =
  \bigdual{J_q(r'),v'}_{p,q} 
  = \norm{r'}_q^{2-q} \Bigdual{ |r'|^{q-1} \sign(r')\,,\, v' }_{p,q} \,.
\end{alignat}
Moreover, in this setting, it is not difficult to show that residual minimization in $[W^{1,q}_{0,\{1\}}(\Omega)]^*$ now coincides with finding the \emph{best $L^p$-approximation} to~$\sign (x)$. The Gibbs phenomena for best $L^p$-approximation was studied analytically by Saff~\& Tashev~\cite{SafTasEJA1999}, when using continuous piecewise-linear approximations on $n$~equal-sized elements. They clarified that the overshoot next to a discontinuity \emph{remains} as $n\rightarrow \infty$ whenever~$p>1$, however remarkably, the \emph{overshoot tends to zero as~$p \rightarrow 1^+$}. 
\par
\begin{figure}[!t]
\centering{%
\psfrag{x}{}
\psfrag{Nelem =  4}{\scriptsize $h = 1/2$}
\psfrag{Nelem =  8}{\scriptsize $h = 1/4$}
\psfrag{Nelem = 16}{\scriptsize $h = 1/8$}
\psfrag{Nelem = 32}{\scriptsize $h = 1/16$}
\includegraphics[scale=1]{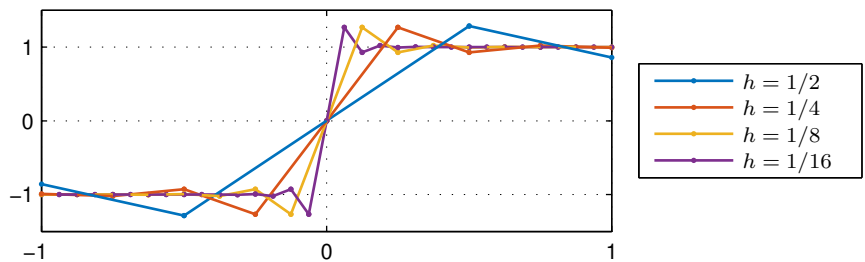}
}
\caption{The best $L^2$-approximation of $u(x) = \operatorname{sign}(x)$ displays the Gibbs phenomenon: The overshoot next to the discontinuity persists on any mesh.}
\label{fig:Gibbs1}
\end{figure}
To illustrate these findings, we plot in Figure~\ref{fig:Gibbs1} the best $L^2$-approximation using continuous piecewise-linears for various mesh-size parameter~$h$. Clearly, the overshoots remain present, signifying the Gibbs phenomenon. Next, in Figure~\ref{fig:Gibbs2} we plot the solution to ideal residual minimization (i.e., in the so-called \emph{ideal} case where $\mbbV_m = \mbbV$) on a fixed mesh consisting of nine elements, for different values of~$p>1$.\footnote{These plots were obtained by using the analytical results by Saff~\& Tashev for the $L^p$-approximation of~$\sign(x)$.} In Figure~\ref{fig:Gibbs2}, we also plot the corresponding ideal residual~$r'(x)$ as defined by the mixed formulation~\eqref{eq:AdvReac_discrete} in the case where $\mbbV_m = \mbbV$. 
It can be shown that in this ideal 1-D situation $r' =\|u_n-u\|_p^{2-p}|u_n-u|^{p-1}\sign(u_n-u)$. The plots in Figure~\ref{fig:Gibbs2} clearly illustrate the elimination of the Gibbs phenomenon as~$p\rightarrow 1^+$. 
\par
We note that the elimination of the Gibbs phenomena was also observed for residual minimization of the \emph{strong form} of the advection--reaction residual in~$L^1(\Omega)$~\cite{GueSINUM2004}, the explanation of which remains somewhat elusive.
\begin{figure}[!t]
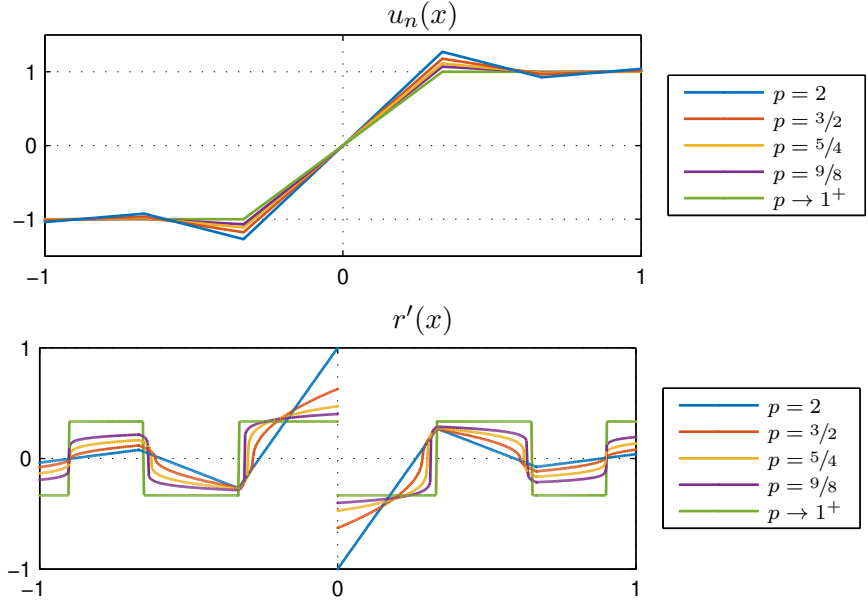

\centering{%
\psfrag{x}{}
\psfrag{p=     2}{\scriptsize $p=2$}
\psfrag{p=   1.5}{\scriptsize $p=\sfrac{3}{2}$}
\psfrag{p=  1.25}{\scriptsize $p=\sfrac{5}{4}$}
\psfrag{p= 1.125}{\scriptsize $p=\sfrac{9}{8}$}
\psfrag{p=     1}{\scriptsize $p\rightarrow 1^+$}
{\small $u_n(x)$}
\\
\includegraphics[scale=1]{data_Gibbs/AdvectionMixedFEM_elemCount6_n1_dnInf_pConv_runs5_u.eps}
\\
{\small $r'(x)$}
\\
\includegraphics[scale=1]{data_Gibbs/AdvectionMixedFEM_elemCount6_n1_dnInf_pConv_runs5_drdx.eps}
}
\caption{Vanishing Gibbs phenomena as~$p\rightarrow 1^+$ for approximations to the discontinuous solution $u(x) = \sign (x)$ given by ideal residual-minimization of weak advection. 
}
\label{fig:Gibbs2}
\end{figure}
\par
Next, consider the \DDResMin{} method~\eqref{eq:AdvReac_discrete} with the discrete space pair $(\mbbU_n,\mbbV_m)$ defined by:
\begin{subequations}
\begin{alignat}{6}
\label{eq:P1cont}
 \mbbU_n 
 &\subseteq \mbbP^1_{\mathrm{cont}}(\mcT_n)
  &&:=  \Big\{ &w_n &\in \mcC^0[-1,1] \,:\,
&&  {w_n}|_{T} \in \mbbP^1(T) \,,\,&& \forall T\in \mcT_n 
  \Big\},
&& 
\\ \notag
  \mbbV_m 
  &\supseteq \mbbP^k_{\mathrm{cont},0,\{1\}}(\mcT_n)
  &&:= \Big\{ &\phi_m &\in \mcC^0[-1,1]  \,:\,
&&  {\phi_m}|_{T} \in \mbbP^k(T)\,,\,&& \forall T\in \mcT_n  \text{ and }
\\  \label{eq:Pkcont}
 &&&&&&&&&  \quad \phi_m(1)=0
  \Big\},
&&   
\end{alignat}
\end{subequations}
where $k$~is the polynomial degree of the test space, and $\mcT_n$ is any partition of the interval $(-1,1)$.
%
\begin{proposition}[1-D advection: Compatible pair]
\label{prop:advContPair}
Let $\mbbU = L^p(-1,1)$, $\mbbV = W^{1,q}_{0,\{1\}}(-1,1)$ and ~$(\mbbU_n,\mbbV_m)$ be defined as above. If~$k\ge 2$, then the Fortin condition (Assumption~\ref{assumpt:Fortin}) holds for the operator $B:\mbbU\rightarrow \mbbV^*$ defined by:
\begin{alignat*}{2}
\bigdual{ B w, v }_{\mbbV^*,\mbbV} =-\int_{-1}^1 u v'\,.
\end{alignat*}
\end{proposition}
%
\begin{proof}
See Section~\ref{sec:advContPairProof}.
\end{proof}
%
\par
%
\begin{remark}[Solution of nonlinear system]
\label{rem:nonlinSolve}
The \DDResMin{} method leads to a discrete (nonlinear) $q$-Laplace mixed system, which, for $p$~moderately close to~$2$, can be solved with, e.g., Newton's or Picard's method. For~$p$ close to~$1$ or much larger than~$2$ the nonlinear problem becomes more tedious to solve, and we have resorted to continuation techniques (with respect to~$p$) or a descent method for the equivalent constrained-minimization formulation.
\end{remark}
%
\par
Figure~\ref{fig:Gibbs3} diplays numerical results obtained using the \DDResMin{} method~\eqref{eq:AdvReac_discrete} with the above discrete spaces and for~$p=1.01$ (hence $q=101$). We plot~$u_n$ and~$r_m'$ for various test-space degree~$k\ge 2$. While the method is stable for any~$k\ge 2$ (owing to Proposition~\ref{prop:advContPair}), there is no reason for the \DDResMin{} method to directly inherit any qualitative feature of the exact residual minimization (i.e., $\mbbV_m=\mbbV$). Indeed, overshoot is present for small~$k$. 
However, results are qualitatively converging once the test space~$\mbbV_m$ starts resolving the ideal~$r$. This is expected by~\cite[Proposition~4.2]{MugZeeARXIV2018}, which states that the ideal~$u_n$ is obtained if the ideal~$r$ happens to be in~$\mbbV_m$. The lines indicated by ``ideal'' in Figure~\ref{fig:Gibbs3} correspond to the case that~$r$ is fully resolved. 
\begin{figure}[!t]
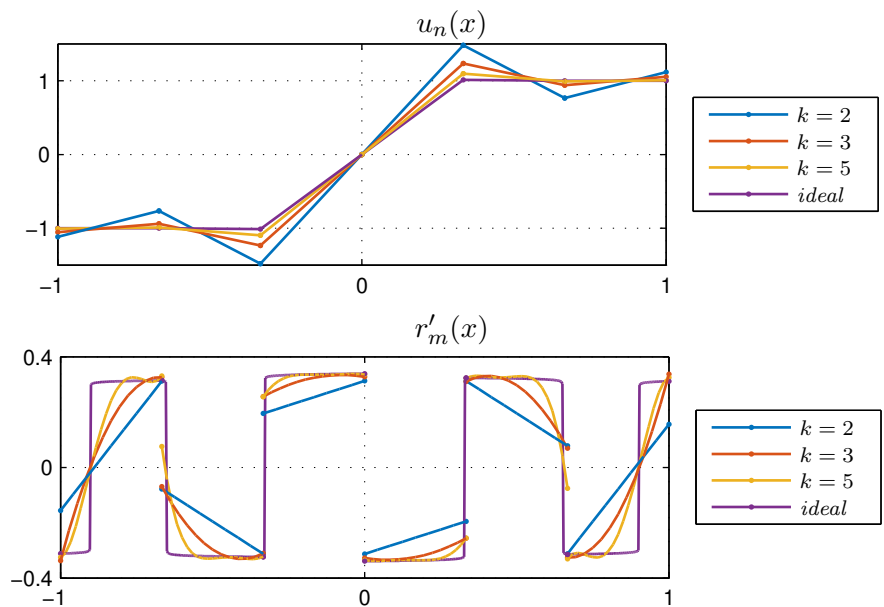

\centering{%
\psfrag{x}{}
\psfrag{dn =   1}{\scriptsize $k=2$}
\psfrag{dn =   2}{\scriptsize $k=3$}
\psfrag{dn =   4}{\scriptsize $k=5$}
\psfrag{dn = Inf}{\scriptsize \emph{ideal}}
{\small $u_n(x)$}
\\
\includegraphics[scale=1]{data_Gibbs/AdvectionMixedFEM_p1_01_elemCount6_n1_dnConv_runs4_u.eps}
\\
{\small $r_m'(x)$}
\\
\includegraphics[scale=1]{data_Gibbs/AdvectionMixedFEM_p1_01_elemCount6_n1_dnConv_runs4_drdx.eps}
}
\caption{Approximations to~$u(x) = \sign(x)$ given by the \DDResMin{} method for weak advection with trial space $\mbbU = L^p(\Omega)$ and~\mbox{$p=1.01$}. The discrete space~$\mbbU_n$ consists of continuous piecewise-linears and $\mbbV_m$ of continuous piecewise-polynomials of degree~$k$. The line \emph{ideal} corresponds to the case that $r\in\mbbV_m$ or $\mbbV_m = \mbbV$ (or $k\rightarrow \infty$).}
\label{fig:Gibbs3}
\end{figure}
\par
Interestingly, the results seem to indicate that different values of~$k$ in each element would be needed for efficiently addressing Gibbs phenomena. This is reminiscent of the idea of \emph{adaptive stabilization}~\cite{CohDahWelM2AN2012}, in which for a given~$\mbbU_n$ a sufficiently large test-space~$\mbbV_m$ is found in an adaptive manner so as to achieve stability. 
%
%
\par

\subsection{The pair~$\mbbP^0(\mcT_n)$~-~$B^{-*}(\mbbP^0(\mcT_n))$: An optimal compatible pair}
\label{sec:optimalTestSpace}
In the remainder of Section~\ref{sec:applications}, we consider for~$\mbbU_n$ piecewise-constant functions on mesh-partitions~$\mcT_n$ of~$\Omega\subset\mathbb{R}^d$, i.e., 
\begin{subequations}
\begin{alignat}{2}\label{eq:P0Un}
  \mbbU_n \subseteq \mbbP^0(\mcT_n) 
  := \Big\{ w_n \in L^\infty(\Omega) \,:\, {w_n}|_T \in \mbbP^0(T) 
  \,,\, \forall T\in \mcT_n \Big\} 
\subset \mbbU.
\end{alignat}
For the discrete test space~$\mbbV_m \subset \mbbV$, we assume that it includes the following \emph{optimal space}~$\mbbS_n := B^{-*} \big( \mbbP^0(\mcT_n) \big) \subset \mbbV$, i.e., 
\begin{alignat}{2}\label{eq:S(T_n)}
 \mbbV_m \supseteq \mbbS_n
 &:= B^{-*} \big( \mbbP^0(\mcT_n) = \Big\{ 
   \phi_n \in \mbbV \,:\, 
   B^* \phi_n = \chi_n\,,\, \hbox{for some } \chi_n\in \mbbP^0(\mcT_n)
 \Big\}.
\end{alignat}
\end{subequations}
Note that $\dim \mbbU_n \le \dim \mbbP^0(\mcT_n) = \dim \mbbS_n \le \dim \mbbV_m$. 
\par
Without any further assumptions, the following striking result show that this pair satisfies the Fortin condition. Its proof hinges on the fact that the $L^p$~ duality map of any $w_n\in \mbbP^0(\mcT_n)$ is also in~$\mbbP^0(\mcT_n)$.
\begin{proposition}[Weak advection--reaction: Compatible pair]
\label{prop:AdvReac_compatible}
Let $\mbbU = L^p(\Omega)$, $\mbbV = W^{q}_{0,+}(\bsbeta;\Omega)$ and the discrete pair~$(\mbbU_n,\mbbV_m)$ be defined as in~\eqref{eq:P0Un} and~\eqref{eq:S(T_n)}. Then the Fortin condition (Assumption~\ref{assumpt:Fortin}) holds for~$B$ defined in~\eqref{eq:B_AdvReac}, with $C_\Pi =  M_\mu/\gamma_B$, where $M_\mu$ is the continuity constant of~$B$ (see Remark~\ref{rem:mu_continuity}), and $\gamma_B$ the bounded-below constant of~$B$ (see Theorem~\ref{thm:avdreact_wellposed}).
\end{proposition}
%
\begin{proof}
In view of the equivalence between the discrete inf-sup condition
and the Fortin condition (see Ern \&~Guermond~\cite{ErnGueCR2016}), 
we prove this proposition by directly establishing the discrete inf-sup condition. 
\par
%
Let~$w_n \in \mbbU_n\subseteq\mbbP^0(\mcT_n)$, then
\begin{alignat}{2}
\notag
  \sup_{v_m\in \mbbV_m} \frac{\dual{B w_n, v_m}_{\mbbV^*,\mbbV}}{ \norm{v_m}_\mbbV }
 &\ge 
  \sup_{\phi_n\in \mbbS_n} \frac{\dual{B w_n, \phi_n}_{\mbbV^*,\mbbV}}{ \norm{\phi_n}_\mbbV }
=  
   \sup_{\chi_n\in \mbbP^0(\mcT_n)} 
    \frac{\dual{w_n, \chi_n}_{p,q}}{ \norm{B^{-*}\chi_n}_\mbbV}\,.
\end{alignat}
Let $J_p(w_n) := \norm{w_n}_p^{2-p} |w_n|^{p-1} \sign (w_n)$ denote the $L^p$~duality map of~$w_n$, and notice that it is also in~$\mbbP^0(\mcT_n)$. Furthermore, we have the duality-map property $\dual{w_n, J_p(w_n)}_{p,q}=\|w_n\|_p\|J_p(w_n)\|_q\,$. Therefore,
\begin{alignat}{2}
\notag
 \sup_{\chi_n\in \mbbP^0(\mcT_n)} 
    \frac{\dual{w_n, \chi_n}_{p,q}}{ \norm{B^{-*}\chi_n}_\mbbV }
 &\ge
    \frac{\dual{w_n, J_p(w_n)}_{p,q}}{ \norm{B^{-*}J_p(w_n)}_\mbbV }
=
    \frac{\norm{w_n}_p \norm{J_p(w_n)}_q}{ \norm{B^{-*}J_p(w_n)}_\mbbV } 
\ge
    \gamma_B \norm{w_n}_p \,,
\end{alignat}
where, in the last step, we used that $\norm{B^{-*} \chi}_\mbbV \le \gamma_B^{-1} \norm{\chi}_{q}$ for all $\chi\in L^q(\Omega)$ (this is nothing but the dual counterpart of Theorem~\ref{thm:avdreact_wellposed}).
Finally, \cite[Theorem~1]{ErnGueCR2016} implies the existence of a Fortin operator $\Pi:\mbbV\to\mbbV_m$ 
with $C_\Pi = M_\mu/\gamma_B$.
%
\end{proof}
%
\begin{remark}[Petrov--Galerkin method]
\label{rem:PG}
If~$(\mbbU_n,\mbbV_m) \equiv (\mbbP^0(\mcT_n) ,\mbbS_n)$, then $\dim \mbbU_n = \dim \mbbV_m$. Then, Proposition~\ref{prop:AdvReac_compatible} together with~\eqref{eq:AdvReac_discrete_b} implies that $r_m = 0$. Thus we obtain from~\eqref{eq:AdvReac_discrete_a} that the approximation~$u_n$ satisfies the Petrov--Galerkin statement (cf.~\cite[Section~5]{MugZeeARXIV2018}):
\begin{equation} \label{eq:PG}
\dual{B u_n,v_n}_{\mbbV^*,\mbbV} = \dual{ f, v_n}_{\mbbV^*,\mbbV}
\qquad  \forall v_n \in \mbbS_n\,.
\end{equation} 
\end{remark}
\begin{remark}[Cell average]
\label{rem:cellAve}
If~$(\mbbU_n,\mbbV_m) \equiv (\mbbP^0(\mcT_n) ,\mbbS_n)$, the approximation~$u_n$ is in fact the element average of the exact solution~$u$, i.e.,
\begin{alignat}{2}
\label{eq:elemAve}
  {u_n}|_T = |T|^{-1}\int_T u
  \,,
  \qquad \forall T\in \mcT_n\,.
\end{alignat}
To prove~\eqref{eq:elemAve}, note that~\eqref{eq:PG} can be written as
\begin{equation*}
 \dual{u_n,B^* v_n} = \dual{u,B^* v_n} \qquad \forall v_n \in \mbbS_n\,.
\end{equation*}
 Let $\chi_T$ be the characteristic function of the element $T$. Then, the test function~$v_T = B^{-*} \chi_{T}$ determines ${u_n}|_T$. Indeed,
\begin{alignat*}{2}
|T|\, {u_n}|_T  = \dual{u_n,\chi_T}_{p,q} 
 = \dual{u_n,B^* v_T}_{p,q} = \dual{u, B^* v_T}_{p,q} =  \dual{u, \chi_T}_{p,q} 
 =  \int_T u \,.
\end{alignat*}
\end{remark}
%
\begin{remark}[Quasi-uniform meshes]
\label{rem:advUnifConv}
In the case that $\mbbU_n = \mbbP^0(\mcT_n)$ where the partitions~$\{\mcT_n\}$ are quasi-uniform shape-regular meshes with mesh-size parameter~$h$, the following a~priori error estimate is immediate (apply Remark~\ref{rem:advFEM} with~$k=0$), provided that $u\in W^{s,p}(\Omega)$ for $0\le s\le 1$:
\begin{alignat*}{2}
 \norm{u-u_n}_p \lesssim h^s \snorm{u}_{W^{s,p}(\Omega)}  \,.
\end{alignat*}
\end{remark}
%
\begin{example}[Quasi-optimality: Solution with jump discontinuity]
\label{ex:qOptJumpSol}
To illustrate the convergence of approximations given by the compatible pair~$(\mbbU_n, \mbbV_m) \equiv (\mbbP^0(\mcT_n) , \mbbS_n)$ for~$\Omega \equiv (0,1)$ on uniform meshes using $n=2,4,8,\ldots$~elements of size~$h=1/n$, consider the following exact solution with jump discontinuity (never aligned with the mesh):%
\footnote{%
The approximations are given by~\eqref{eq:elemAve}, or can be obtained by solving the nonlinear discrete problem (see Remark~\ref{rem:nonlinSolve}).
}
\begin{alignat*}{2}
 u(x) = \operatorname{sign}\big(x-\tfrac{\sqrt{2}}{2}\big)\,, \qquad \text{for } x\in (0,1)\,.
\end{alignat*}
It can be shown (e.g.~by computing the Sobolev--Slobodeckij norm) that~$u \in W^{s,p}(0,1)$ for any $0<s<1/p$, but not~$s=1/p$. Figure~\ref{fig:cellAverage} shows the convergence of~$\norm{u-u_n}_p$ with respect to~$h$, for various~$p$. The observed convergence behavior, as anticipated in Remark~\ref{rem:advUnifConv}, is indeed close to~$O(h^{1/p})$.
%
%
%
\end{example}
%
\begin{figure}[!t]
\centering
\psfrag{err p=2.0}{\footnotesize $\norm{u-u_n}_2$}
\psfrag{err p=1.5}{\footnotesize $\norm{u-u_n}_{\sfrac{3}{2}}$}
\psfrag{err p=1.0}{\footnotesize $\norm{u-u_n}_1$}
\psfrag{N\^-1/2.0}{\footnotesize $O(N^{-1/2})$}
\psfrag{N\^-1/1.5}{\footnotesize $O(N^{-2/3})$}
\psfrag{N\^-1/1.0}{\footnotesize $O(N^{-1})$}
\psfrag{N}{\footnotesize Number of elements~$N$}
\psfrag{h\^+1/2.0}{\footnotesize $O(h^{\sfrac{1}{2}})$}
\psfrag{h\^+1/1.5}{\footnotesize $O(h^{\sfrac{2}{3}})$}
\psfrag{h\^+1/1.0}{\footnotesize $O(h)$}
\psfrag{h}{\footnotesize $h$}
\includegraphics[scale=1]{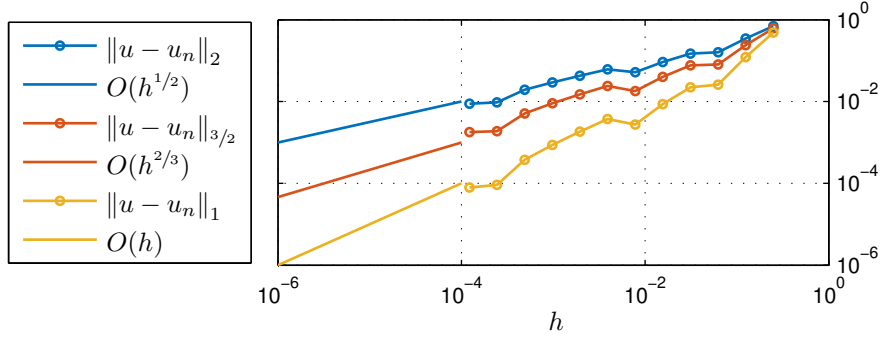}
\caption{Approximating an advection--reaction problem with a discontinuous solution using the optimal pair ($\mbbP^0(\mcT_n),\mbbS_n$): The convergence in~$\norm{u-u_n}_p$ is close to~$O(h^{1/p})$, which is optimal for near-best approximations.}
\label{fig:cellAverage}
\end{figure}
\begin{example}[Basis for optimal test space]
Let us illustrate the discrete test space $\mbbS_n$ in 1-D for the particular case where the (scalar-valued) advection~$\beta(x)$ is space-dependent and $\mu\equiv 0$. Let $\Omega=(0,1)$ and let $\beta$ be a strictly decreasing and positive function such that $\beta'(x)$ is bounded away from zero (hence Assumption~\ref{assump:mu_0} is valid). The space~$\mbbV$ is given by 
\begin{alignat*}{2}
\mbbV= \Big\{v\in L^q(0,1): (\beta v)'\in L^q(0,1) \text{ and } v(1)=0\Big\}.
\end{alignat*}
Let $0=x_1<x_1<\dots<x_{n+1}=1$ be a partition of $\Omega$ and 
define $\mcT_n=\{T_j\}$ where $T_j=(x_{j-1},x_j)$. Let $\chi_{T_j}$ be the characteristic function of $T_j$ and $h_j=|T_j|$. The discrete test space $\mbbS_n$ is defined as the span of the functions $v_j\in\mbbV$ such that $-(\beta v_j)'=\chi_{T_j}$, which upon integrating over the interval $[x,1]$ gives :
\begin{alignat*}{2}
v_j(x)= & \left\{
\begin{array}{ll}
h_j/\beta(x) & \hbox{if } x\leq x_{j-1}\\
(x_j-x)/\beta(x) & \hbox{if } x\in T_j\\
0 & \hbox{if } x\geq x_{j}\,.
\end{array}\right. 
\end{alignat*}
Moreover, we can combine them in order to produce the local, nodal basis functions:
\begin{equation}\nonumber
\widetilde v_1(x)= \beta(x_1){v_1(x)\over h_1}  \quad\hbox{and}
\quad \widetilde v_j(x)= \beta(x_j)\left({v_j(x)\over h_j}-{v_{j-1}(x)\over h_{j-1}}\right), \quad j\geq 2.
\end{equation}
See Figure~\ref{fig:inflowoutflowcont}(a) for an illustration of these basis functions with~$h_j = 0.2$ for all~$j$ and~$\beta(x) = 1.001-x$.
\begin{figure}[!t]
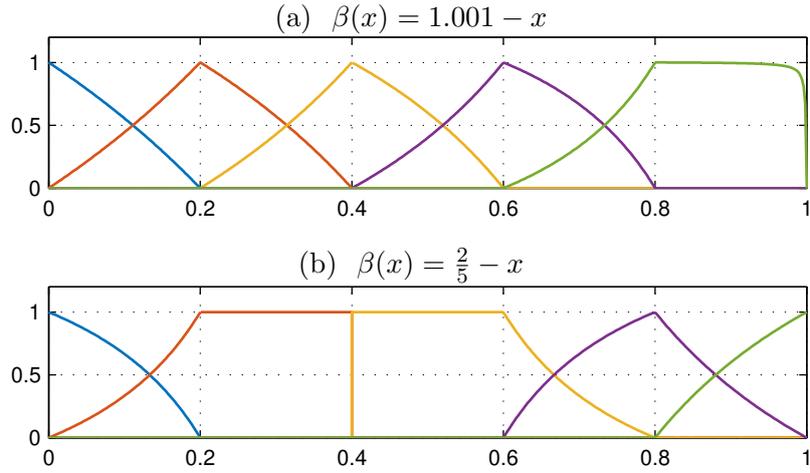

\psfrag{x}{}
\centerline{\small (a)~ $\beta(x) = 1.001 - x$}
\centerline{\includegraphics[scale=1]{data_advBasis/advInflowOutflow_BasisNrm_noLegend.eps}}
\vspace*{1ex}
\centerline{\small (b)~ $\beta(x) = \tfrac{2}{5} - x$}
\centerline{\includegraphics[scale=1]{data_advBasis/advInflowInflow_BasisNrm_noLegend.eps}}
\caption{Basis for the optimal test space~$\mbbS_n$ that is compatible with~$\mbbU_n = \mbbP^0(\mcT_n)$, in the case of two different space-dependent advection fields~$\beta(x)$ corresponding to (a)~left-sided inflow and (b)~two-sided inflow, respectively.}
\label{fig:inflowoutflowcont}
\end{figure}
%
%
%
\par
Another interesting example is when we have two inflows, each on one side of the interval $\Omega=(0,1)$. This is possible by means of a strictly decreasing $\beta(x)$ such that $\beta(0)>0$ and $\beta(1)<0$, and such that $\beta'(x)$ is bounded away from zero. The solution $u\in L^{p}(\Omega)$ of problem~\eqref{eq:weak_advection-reaction} may be singular at the point $\tilde x\in \Omega$ for which $\beta(\tilde x)=0$, even for smooth right hand sides. The test functions computed by solving $-(\beta v_j)'=\chi_{T_j}$ may be discontinuous when $\tilde x$ matches one of the mesh points. This is illustrated in Figure~\ref{fig:inflowoutflowcont}(b) for~$\beta(x) = \frac{2}{5}-x\,$.
\end{example}
\begin{example}[A practical alternative to~$\mbbS_n$]
In practise it may not be feasible to explicitly compute a basis for~$\mbbS_n$. Practical alternatives consist of, for example, continuous piecewise polynomials of sufficiently-high degree~$k$ on~$\mcT_n$,
or continuous piecewise linear polynomials on $\mathsf{Refine}_\ell(\mcT_n)$, which is the submesh obtained from the original mesh~$\mcT_n$ by performing~$\ell$ uniform refinements of all elements (see~\cite{BroDahSteMOC2018} for a similar alternative in a DPG setting).
\par
To illustrate the latter alternative for the \DDResMin{} method, consider the domain~$\Omega = (0,1)$, coefficients $\beta(x) = 1- 12 x$ and $\mu(x) = -4$, source~$f_\circ(x) = 0$, and inflow data~$g$ such that the exact solution is $u(x) =|1-12x|^{-\sfrac{1}{3}}$ for all~$x\in \Omega\setminus \{ \sfrac{1}{12} \}$. Note that $u$ has a singularity and that $u\in L^r(\Omega)$ for any~$1\le r< 3$, but not for~$r\ge 3$. 
\par
In  the method, we take~$p=2$, $\mbbU_n = \mbbP^0(\mcT_n)$ and $\mbbV_m = \mbbP^1_{\mathrm{cont}}(\mathsf{Refine}_\ell(\mcT_n))$, where $\mcT_n$ is a mesh of uniform elements of size~$h = 1/n$, and $\mathsf{Refine}_\ell(\mcT_n)$ is an $\ell$-refined submesh with uniform elements of size~$h_{\ell} = h/(2^\ell)$.
\par
Figure~\ref{fig:P0P1Level:hconv} plots the convergence of the $\norm{u-u_n}_2$ versus~$h$ for~$\ell = 1$, $2$ and~$4$ (error plots are actually similar for all~$\ell\ge 1$). We note that~$\ell = 0$ is in general not sufficiently rich, as it leads to a singular matrix for~$h=1/2$, while the results for~$\ell \ge 1$ did not show any instabilities. To anticipate the rate of convergence, note the Sobolev embedding result $W^{s,2}(\Omega) \subset L^r(\Omega)$ for $s\ge \frac{1}{2} - \frac{1}{r}$ and $r\ge 2$. Therefore, one expects a convergence of~$O(h^s)$ with~$s = \frac{1}{6}$, which is indeed consistent with the numerical observation in~Figure~\ref{fig:P0P1Level:hconv}. The oscillations are caused by the singularity location ($x=\sfrac{1}{12}$) being closer to the left or right element edge depending on~$h$. 
\par
To investigate for a fixed mesh with~$h = \sfrac{1}{16}$ the convergence of the obtained approximations~$u_n$ with respect~$\ell$, we consider~$\beta(x)=2-x$, $\mu(x) = 0$ and exact solution~$u(x) = 1+2x$ for $x\in \Omega$. Figure~\ref{fig:P0P1Level:lconv} plots the error~$\norm{u_n{}_{|\infty} - u_n{}_{|\ell}}$ with respect to~$h_\ell = h / (2^\ell)$, where~$u_n{}_{|\infty}$ denotes the ideal approximation ($\mbbV_m = \mbbV$). For this error we observe a rate of convergence~$O(h_\ell^2)$. 

%
%

\end{example}
\begin{figure}[!t]
\psfrag{h}{\small $h$}
\psfrag{k = 1}{\footnotesize ~$\ell = 1$}
\psfrag{k = 2}{\footnotesize ~$\ell = 2$}
\psfrag{k = 4}{\footnotesize ~$\ell = 4$}
\psfrag{rate = 1/6}{\footnotesize ~$O(h^{\sfrac{1}{6}})$}
\centerline{\small $\bignorm{u-u_n}_2$\hspace*{0.1\textwidth}~}
\centerline{%
\includegraphics[width=0.8\textwidth]{data_P0P1l/Error_of_Res-Min_for_Varying_Levels,_p_=_2_VERSION_3.eps}}
\caption{Approximating a singular solution~$u(x) =|1-12x|^{-\sfrac{1}{3}}$ for $x\in (0,1)\setminus \{\sfrac{1}{12}\}$ to the advection--reaction problem with the \DDResMin{} method using $\mbbU_n = \mbbP^0(\mcT_n)$ and $\mbbV_m = \mbbP^1_{\mathrm{cont}}(\mathsf{Refine}_\ell(\mcT_n))$. The convergence is close to $O(h^{\sfrac{1}{6}})$, which is optimal for near-best approximations.}
\label{fig:P0P1Level:hconv}
\end{figure}
\begin{figure}[!t]
\psfrag{h-size}{\small $h_\ell$}
\centerline{\small \hspace*{0.05\textwidth}$\bignorm{u_{n|\infty}-u_{n|\ell}}_2$}
\centerline{%
\includegraphics[width=0.85\textwidth]{data_P0P1l/Error_of_RESMIN_at_Levels_Simple_Beta_Against_hsize.eps}\rule{12pt}{0pt}}
\caption{Convergence of approximation~$u_{n|\ell}$ toward $u_{n|\infty}$ on a fixed mesh $\mcT_n$ for a smooth exact solution, where $u_{n|\ell}$ denotes the approximation obtained by the \DDResMin{} method using $\mbbU_n = \mbbP^0(\mcT_n)$ and $\mbbV_m = \mbbP^1_{\mathrm{cont}}(\mathsf{Refine}_\ell(\mcT_n))$. The observed convergence is $O(h_\ell^2)$.}
\label{fig:P0P1Level:lconv}
\end{figure}

\subsection{The pair $\mbbP^0(\mcT_n)$~-~$\mbbP^1_\mathrm{conf}(\mcT_n)$: An optimal pair in special situations}
\label{sec:2DflowAlignedMesh}
%
\par
As a last application, we consider a special multi-dimensional situation such that the optimal test space~$\mbbS_n$ defined in~\eqref{eq:S(T_n)} reduces to a convenient finite element space. We focus on a \mbox{2-D}~setting, and assume $\Omega \subset \mbbR ^2$ is polygonal and $\mcT_n$ is a 
simplicial mesh (triangulation) of~$\Omega$.
Let~$\mcF_n = \{F\}$ denote all mesh interior faces.%
\footnote{i.e., $\operatorname{length}(F) >0$, and $F = \pd T_1 \cap \pd T_2$ for distinct~$T_1$ and $T_2$ in~$\mcT_n$.} 
Assume that $\mu\equiv 0$, $\div\bsbeta\equiv 0$ and that the hypothesis of Assumption~\ref{ass:omega-filling} is fulfilled. Assume additionally that $\bsbeta$ is piecewise constant on some partition of~$\Omega$, and let the mesh~$\mcT_n$ be compatible with this partition, i.e.,  
\begin{alignat*}{2}
 &\bsbeta|_{T} \in \mbbP^0(T)\times\mbbP^0(T), \qquad &&\forall T\in \mcT_n\,,
 \\
 &\jump{ \bsbeta\cdot \bsn_F }_F = 0, \qquad &&\forall F\in \mcF_n\,.
\end{alignat*}
where~$\jump{\cdot} = (\cdot)_+ - (\cdot)_-$ denotes the jump.  
Finally, assume that the mesh is \emph{flow-aligned} in the sense that each triangle~$T\in \mcT_n$ has exactly one tangential-flow face~$F \subset \pd T$ for which $\bsbeta \cdot \bsn_T = 0$ on~$F$. Necessarily, the other two faces of~$T$ correspond to in- and out-flow on which $\bsbeta|_T \cdot \bsn_T < 0$ and $\bsbeta|_{T} \cdot \bsn_T > 0$, respectively.
\par
The main result for this special situation is the following characterization of~$\mbbS_n$:
%
\begin{proposition}[Optimal space~$\mbbS_n$: Flow-aligned case]
Under the above assumptions, 
\begin{alignat*}{2}
 \mbbS_n = \mbbP^1_{\mathrm{conf}}(\mcT_n) 
 := \Big\{
   \phi_n \in \mbbV = \Gqplus
   \,:\,
   {\phi_n}|_{T} \in \mbbP^1(T)\,,\, \forall T \in \mcT_n
 \Big \}\,.
\end{alignat*}
\end{proposition}
%
Note that~$\mbbP^1_{\mathrm{conf}}(\mcT_n)$ consists of $\Gqplus$-conforming, piecewise-linear functions, which can be discontinuous across tangential-flow faces, but must be continuous across the other faces. Furthermore, they are zero on~$\pd\Omega_+$.
%
\begin{proof}
The proof follows upon demonstrating that~$B^*\mbbP^1_{\mathrm{conf}}(\mcT_n) = \mbbP^0(\mcT_n)$.
First note (under the above assumptions) that  $B^* = -\bsbeta\cdot \grad_n$, where~$\grad_n$ is the element-wise (or broken) gradient, i.e., 
$(\grad_n \phi)|_{T} = \grad(\phi|_{T})$ for all~$T\in \mcT_n$. Since functions in~$\mbbP^1_{\mathrm{conf}}(\mcT_n)$ are element-wise linear, we thus have that $B^* \mbbP^1_{\mathrm{conf}}(\mcT_n) \subset \mbbP^0(\mcT_n)$.
\par
We next show that $\mbbP^0(\mcT_n)\subset B^* \mbbP^1_{\mathrm{conf}}(\mcT_n)$. Note that $\mbbP^0(\mcT_n) = \Span \{\chi_T , T\in \mcT_n\}$, where~$\chi_T$ is the characteristic function for~$T$. Let~$\phi_T$ be the unique solution in~$\mbbV$ such that~$B^* \phi_T = \chi_T$. The $\Omega$-filling assumption (see Assumption~\ref{ass:omega-filling}) guarantees that $\bsbeta \neq \bs 0$ a.e.~in $\Omega$ (otherwise we would have $-\bsbeta\cdot \nabla z_\pm=0$ in some element, contradicting~\eqref{eq:omega-filling}).
Thus, for a.e.~$x\in\Omega$ consider the polygonal path $\Gamma(x)\subset \overline\Omega$ that starts from $x$ and moves along the advection field $\bsbeta$. By the $\Omega$-filling assumption, the path $\Gamma(x)$ has to end in some point on the out-flow boundary $\partial\Omega_+$ (otherwise it will stays forever within $\Omega$, contradicting the existence of a bounded function $z_\pm\in W^\infty(\bsbeta;\Omega)$ whose absolute value grows linearly along $\Gamma(x)$).  Hence, we can construct $\phi_T$ integrating $\chi_T$ over the polygonal path $\Gamma(x)$ from $\partial\Omega_+$ to $x$. By construction, $\phi_T$ is a piecewise linear polynomial, which can be discontinuous only across $\{F\in\mcF_n: \bsbeta\cdot\bsn_F=0 \}$. Besides, $\phi_T$ satisfies the homogeneous boundary condition over $\partial\Omega_+$. Hence $\phi_T\in \mbbP^1_{\mathrm{conf}}(\mcT_n)$ and 
$\chi_T\in B^*\mbbP^1_{\mathrm{conf}}(\mcT_n)$.
%
%
\end{proof}
%
\begin{example}[2-D numerical illustration]
To illustrate the above setting with a numerical example, let $\Omega = (0,1)\times (0,2)\subset \mathbb{R}^2$, $f_\circ = 0$, and $g$ be nonzero on the inflow boundary~$\pd\Omega_- = \{(x,0)\,,\, x\in (0,1)\}$. Let an initial triangulation of the domain be as in Figure~\ref{fig:2Dsmooth} (top-left mesh). The advection~$\bsbeta$ is such that, for the bottom, left, right and top boundary, we have that $\bsbeta \cdot \bsn$ is~$-1$, $0$, $0$ and~$1$, respectively. Next, within each triangle, $\bsbeta$ is some constant vector with a positive vertical component, while satisfying the above requirements (i.e., $\jump{\bsbeta \cdot \bsn_F}_F = 0$ on each interior face~$F$, and each triangle has a tangential-flow, in-flow and out-flow face).%
\footnote{For a given mesh, such a~$\bsbeta$ can be constructed by traversing through the mesh in an element-by-element fashion, starting at the inflow boundary, and assigning $\bsbeta$ in each element so as to satisfy the requirements.} 
By Remark~\ref{rem:PG}, the \DDResMin{} method with spaces~$\mbbP^0(\mcT_n)$ and $\mbbP^1_{\mathrm{conf}}(\mcT_n)$ can be implemented as a Petrov--Galerkin method. 
\par
We first consider the smooth inflow boundary condition~$g(x,0) = \sin (\pi x)$ for~$x\in (0,1)$. Figure~\ref{fig:2Dsmooth} (top row) shows the approximations for~$u_n$ obtained on the initial triangulation and three finer meshes. The finer meshes were obtained by uniform refinements of the initial triangulation using so-called \emph{red}-refinement~\cite[Section~2.1.2]{VerBOOK2013} (splitting each triangle into four similar triangles), which preserves the above flow-aligned mesh requirement. The approximations nicely illustrate the cell-average property mentioned in Remark~\ref{rem:cellAve} (the exact solution is simply found by traversing $g$ along the characteristics). In Figure~\ref{fig:2Dsmooth} (bottom) the convergence of~$\norm{u-u_n}_p$ is shown to be optimal (rate is $O(h)$) for various values of~$p$. 
\par
Figure~\ref{fig:2Dnonsmooth} shows the same results as before, but now for a discontinuous inflow boundary condition~$g(x,0) = \sin(\pi x) \, \sign(x-\sfrac{1}{3})$ for~$x\in (0,1)$. Again the \DDResMin{} method provides a near-best approximation; as anticipated, the observed rate of convergence is~$O(h^{1/p})$ (cf.~discussion in Example~\ref{ex:qOptJumpSol}).
\end{example}

\newcommand{\includetwodfig}[1]{{\includegraphics[height=0.4\textwidth,viewport=0 0 160 340,clip]{#1}}}
\newcommand{\includetwodfigcb}[1]{\raisebox{-0.005\textwidth}{{\includegraphics[height=0.4\textwidth,viewport=160 0 220 340,clip]{#1}}}}

\begin{figure}[!t]
\centerline{%
\includetwodfig{data_2D/smooth/u_=_sin(pix)_Approximation_RL0.eps}
~~
\includetwodfig{data_2D/smooth/u_=_sin(pix)_Approximation_RL1.eps}
~~
\includetwodfig{data_2D/smooth/u_=_sin(pix)_Approximation_RL2.eps}
~~
\includetwodfig{data_2D/smooth/u_=_sin(pix)_Approximation_RL3.eps}
~~
\includetwodfigcb{data_2D/smooth/u_=_sin(pix)_Approximation_RL0.eps}
}
\psfrag{h}{$h$}
\psfrag{p = 1}{\footnotesize ~$p = 1$}
\psfrag{p = 1.5}{\footnotesize ~$p = \sfrac{3}{2}$}
\psfrag{p = 2}{\footnotesize ~$p = 2$}
\psfrag{p = 3______}{\footnotesize ~$p = 3$}
~\\
\centerline{\small $\norm{u-u_n}_p$}
\centerline{\includegraphics[width=1\textwidth]{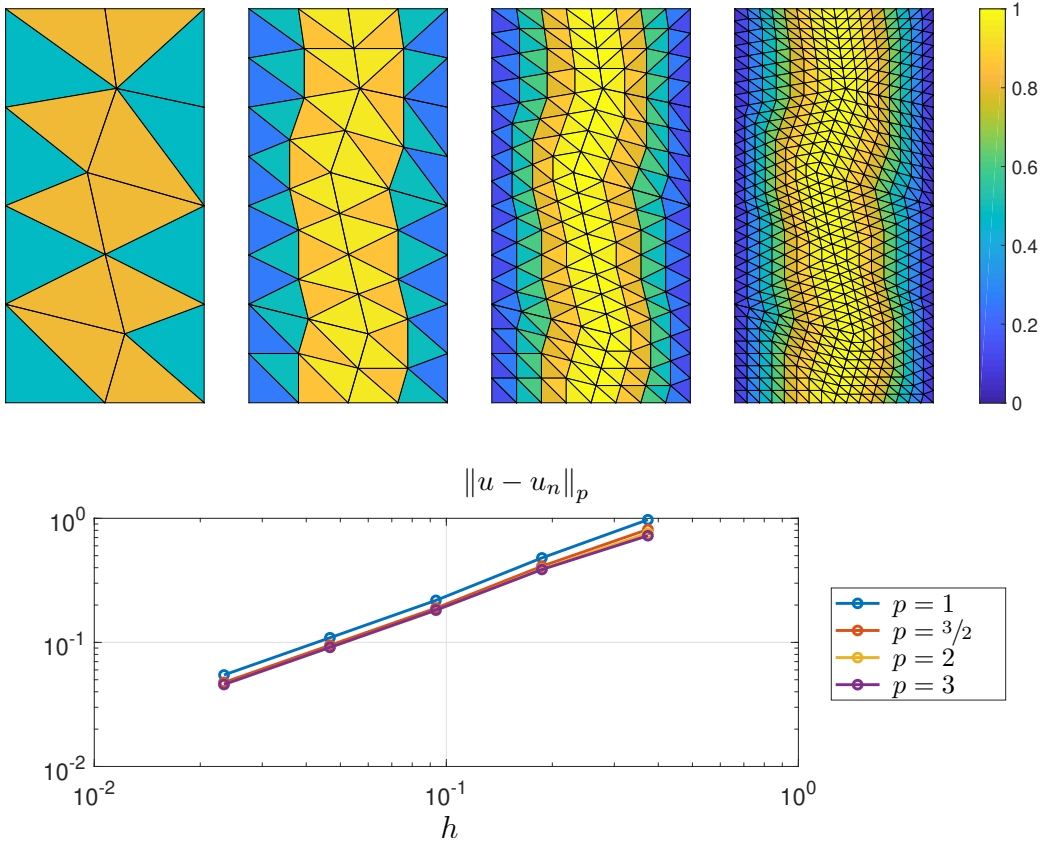}}
\caption{\DDResMin{} approximations using~$\mbbU_n = \mbbP^0(\mcT_n)$ and~$\mbbV_m = \mbbP^1_{\mathrm{conf}}(\mcT_n)$ for an incompressible advection problem with special piecewise constant~$\bsbeta$ and a \emph{smooth} inflow boundary condition~$g$. \emph{Top row}: Approximations~$u_n$ on different (nested) meshes. \emph{Bottom}: The convergence in $\norm{u-u_n}_p$ is~$O(h)$ (optimal).}
\label{fig:2Dsmooth}
\end{figure}


\renewcommand{\includetwodfig}[1]{{\includegraphics[height=0.4\textwidth,viewport=0 0 160 340,clip]{#1}}}
\renewcommand{\includetwodfigcb}[1]{\raisebox{-0.005\textwidth}{{\includegraphics[height=0.4\textwidth,viewport=160 0 220 340,clip]{#1}}}}

\begin{figure}[!t]
\centerline{%
\includetwodfig{data_2D/nonsmooth/u_=_sin(pix)sign(x-x0)_Approximation_RL0.eps}
~~
\includetwodfig{data_2D/nonsmooth/u_=_sin(pix)sign(x-x0)_Approximation_RL1.eps}
~~
\includetwodfig{data_2D/nonsmooth/u_=_sin(pix)sign(x-x0)_Approximation_RL2.eps}
~~
\includetwodfig{data_2D/nonsmooth/u_=_sin(pix)sign(x-x0)_Approximation_RL3.eps}
~~
\includetwodfigcb{data_2D/nonsmooth/u_=_sin(pix)sign(x-x0)_Approximation_RL0.eps}
}
\psfrag{h}{$h$}
\psfrag{p = 1}{\footnotesize $p = 1$}
\psfrag{p = 1.5}{\footnotesize $p = \sfrac{3}{2}$}
\psfrag{p = 2}{\footnotesize $p = 2$}
\psfrag{p = 3______}{\footnotesize $p = 3$}
~\\
\centerline{\small $\norm{u-u_n}_p$}
\centerline{\includegraphics[width=1\textwidth]{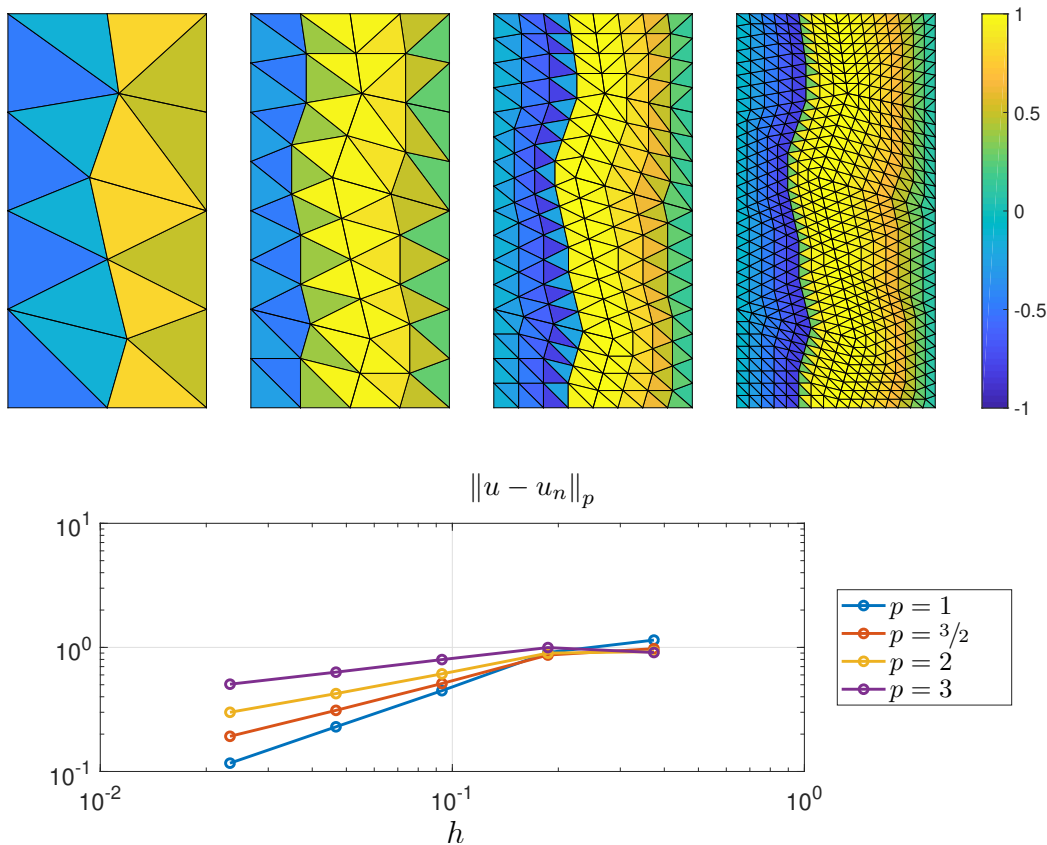}}
\caption{\DDResMin{} approximations using~$\mbbU_n = \mbbP^0(\mcT_n)$ and~$\mbbV_m = \mbbP^1_{\mathrm{conf}}(\mcT_n)$ for an incompressible advection problem with special piecewise constant~$\bsbeta$ and a \emph{discontinuous} inflow boundary condition~$g$. \emph{Top row}: Approximations~$u_n$ on different (nested) meshes. \emph{Bottom}: The convergence in $\norm{u-u_n}_p$ is~$O(h^{1/p})$, which is optimal for near-best approximations to the discontinuous solution~$u$.}
\label{fig:2Dnonsmooth}
\end{figure}

%
%
%

%

\appendix
\section{Proofs of the main results}
\subsection{Proof of Theorem~\ref{thm:avdreact_wellposed}}\label{sec:avdreact_wellposed}
In this section, we give the proof of Theorem~\ref{thm:avdreact_wellposed} by means of the so-called \emph{Banach-Ne\v{c}as-Babu\v{s}ka $\inf$-$\sup$ conditions} (see~\cite[Theorem~2.6]{ErnGueBOOK2004}):
\begin{alignat}{2}
\tag{BNB1}\label{eq:BNB1}
& \|w\|_\mbbU \lesssim \sup_{0\neq v\in \mbbV}{|b(w,v)|\over \|v\|_\mbbV},\,\, \forall w\in\mbbU,\\
\tag{BNB2}\label{eq:BNB2}
& \big\{ v\in\mbbV : b(w,v)=0, \,\,\forall  w \in \mathbb U\big\} = \{0\}. 
\end{alignat}
Our technique is similar to the one used by Cantin~\cite{CanCR2017}, but note that we prove~\eqref{eq:BNB1}-\eqref{eq:BNB2} on the adjoint bilinear form. Recall that the primal operator is a continuous bijection if and only if the adjoint operator is a continuous bijection, in which case both inf-sup constants are the same.
The following proof is also analogue to the proof in Hilbert spaces given by Di~Pietro~\& Ern~\cite[Section 2.1]{DipErnBOOK2012}. We start by giving some properties that we need for the Banach setting.

Let $J_q(v)=\|v\|_q^{2-q} |v|^{q-1}\sign(v)\in L^p(\Omega)=\mathbb U$ be the duality map of $L^q(\Omega)$, i.e., 
\begin{alignat}{2}\label{eq:J_q}
\<J_q(v),v\>_{p,q}=\|v\|_q^2 \quad \hbox{ and } \quad \|J_q(v)\|_p=\|v\|_q,\qquad\forall v\in L^q(\Omega). 
\end{alignat} 
Additionally, for any $v\in \mathbb V=\Gqplus
\subset L^q(\Omega)$ notice the following identity:
\begin{equation}\label{eq:beta_identity2}
\bsbeta\cdot\nabla v\,\,|v|^{q-1}\sign(v)={1\over q}\div(\bsbeta|v|^q)-{1\over q}\div(\bsbeta)|v|^q,\quad\forall v\in \mathbb V.
\end{equation}
We will use these definitions and properties also for their analogous ``$p$'' version, i.e., obtained by replacing $q$ by~$p$.


\subsubsection{Proof of $\inf$-$\sup$ condition~\eqref{eq:BNB1} on the adjoint}

Let $b:\mathbb U\times\mathbb V\rightarrow\mbbR$ be the bilinear form corresponding to the weak--form in~\eqref{eq:weak_advection-reaction}, i.e., 
\begin{alignat}{2}\label{eq:AdvRea_bform}
  b(w,v) = \int_\Omega w \big( \mu v - \div (\bsbeta v) \big)\,.
\end{alignat}
For any $0\neq v\in \mathbb V$ we have:
\begin{alignat*}{2}
\notag
\sup_{0\neq w\in\mathbb U}{|b(w,v)|\over \|w\|_p} &\geq 
 {|b(J_q(v),v)|\over \|J_q(v)\|_p}
\\
\tag{by~\eqref{eq:J_q}~and~\eqref{eq:AdvRea_bform}}
 &= 
 \|v\|_q^{1-q}\left|\int_\Omega|v|^{q-1}\sign(v)\left(\mu v-\div(\bsbeta v)\right) \right|
 \\
\tag{by~\eqref{eq:beta_identity}}
&= 
\|v\|_q^{1-q}\left|\int_\Omega|v|^{q-1}\sign(v)\left(\mu v-\div(\bsbeta)v-\bsbeta\cdot\nabla v\right) \right| 
\\
\tag{by~\eqref{eq:beta_identity2}}
&=    
\|v\|_q^{1-q}\left|\int_\Omega|v|^{q}\left(\mu-{1\over p}\div(\bsbeta)\right)-{1\over q}\div(\bsbeta|v|^q)\right|
\\
\tag{by~\eqref{eq:beta-mu}}
& \geq 
\mu_0\|v\|_q + \|v\|_q^{1-q}\,{1\over q}\int_{\partial\Omega^-}|\bsbeta\cdot\bsn||v|^q
\\
& \geq \mu_0\|v\|_q\,.
\end{alignat*}
Hence, we obtain control on $v$ in the $\norm{\cdot}_q$-norm.\footnote{This result is an extension of the 1-D result with constant advection in~\cite[Chapter~XVII~A, \S{}3, Section~3.7]{DauLioBOOK1992}.}
To control the entire graph norm~$\enorm{\cdot}_{q,\bsbeta}$, we also need to control the divergence part:
\begin{alignat*}{2}
\tag{by duality}
\|\div(\bsbeta v)\|_q &= \displaystyle\sup_{0\neq w\in\mathbb U}
{\left<w,\div(\bsbeta v)\right>_{p,q}\over \|w\|_p}
\\
\tag{by~\eqref{eq:AdvRea_bform}}
&=  \displaystyle\sup_{0\neq w\in\mathbb U}
{\left|b(w,v)-\int_\Omega\mu w v\right|\over \|w\|_p}
\\
\notag
&  \leq \displaystyle\sup_{0\neq w\in\mathbb U}
{|b(w,v)|\over \|w\|_p}+\sup_{0\neq w\in\mathbb U}{\left|\int_\Omega\mu w v\right|\over \|w\|_p}
\\
\tag{by Cauchy--Schwartz ineq.}
&\leq \displaystyle\sup_{0\neq w\in\mathbb U}
{|b(w,v)|\over \|w\|_p}+\|\mu\|_\infty\|v\|_q 
\\
\tag{using the previous bound}
&\leq  \displaystyle\left(1+{\|\mu\|_\infty\over\mu_0}\right)\sup_{0\neq w\in\mathbb U}
{|b(w,v)|\over \|w\|_p}.
\end{alignat*}
Combining both bounds we have
\begin{equation}\label{eq:inf-sup_adjoint}
\enorm{v}_{q,\bsbeta}\leq {\sqrt{1+(\mu_0+\|\mu\|_\infty)^2\over\mu_0^2}}
\sup_{0\neq w\in\mathbb U}
{|b(w,v)|\over \|w\|_p}.
\end{equation}

The case when $\mu\equiv 0$ and $\div\bsbeta\equiv 0$ (under Assumption~\ref{ass:omega-filling}) is simpler, since by Lemma~\ref{lem:Poincare_Friedrichs}, we immediately have
\begin{alignat}{2}\notag
\|v\|_q \,\leq\, \CPF
\|\bsbeta\cdot\nabla v\|_q \,=\,  \CPF
\|\div(\bsbeta v)\|_q \,=\, \CPF
\sup_{0\neq w\in\mathbb U}
{|b(w,v)|\over \|w\|_p} \,.
\end{alignat} 
Hence
\begin{equation}\label{eq:inf-sup_adjoint_mu=0}
\enorm{v}_{q,\bsbeta}\leq (1+\CPF)
\sup_{0\neq w\in\mathbb U}
{|b(w,v)|\over \|w\|_p}\,.
\end{equation} 
\hfill%
$\ensuremath{_\blacksquare}$
\subsubsection{Proof of $\inf$-$\sup$ condition~\eqref{eq:BNB2} on the adjoint}
\label{sec:surjec_adjoint}
Next, we prove~\eqref{eq:BNB2} for the adjoint, which corresponds to injectivity of the primal operator. In other words, we need to show that $w=0$ if $w\in L^p(\Omega)$ is such that
\begin{alignat}{2}
\label{eq:adv-infsup2}
  b(w,v)=0, \qquad \forall v\in \mathbb V=\Gqplus\,.
\end{alignat}
\par
We first take $v\in C_0^\infty(\Omega)$ to obtain that 
$\bsbeta\cdot\nabla w + \mu w =0$ in the sense of distributions and hence
$\bsbeta\cdot\nabla w = -\mu w \in L^p(\Omega)$, which implies $w\in \Gp$. 
\par
This means that $w$ has sufficient regularity so that traces make sense (see Remark~\ref{rem:traces}). Hence, going back to~\eqref{eq:adv-infsup2} and integrating by parts we have:
\begin{equation}\label{eq:vanish_inflow}
\int_{\partial\Omega^-}{\bsbeta\cdot\bsn}\,w\,v=0\, ,
\qquad\forall v \in \Gqplus.
\end{equation} 
To show that $w\in \Gpminus$, we consider 
$\widetilde J_p(w):=|w|^{p-1}\sign(w)\in L^q(\Omega)$. The fact that $\widetilde J_p(w)$ is actually in~$\Gq$ is proven in Lemma~\ref{lem:beta_nabla} below. For the function $\phi\in C^\infty(\overline\Omega)$ defined in~\eqref{eq:cutoff}, we then have that $\phi\,\widetilde J_p(w)$ belongs to $\Gqplus$ (since $\phi$ vanishes on $\partial\Omega_+$).
%
Using $v=\phi\,\widetilde J_p(w)$ in~\eqref{eq:vanish_inflow} we immediately obtain:
\begin{equation}\label{eq:zero_trace}
\int_{\partial\Omega^-}{\bsbeta\cdot\bsn}\,\,|w|^p=0\,,
\end{equation}
hence that~$w\in \Gpminus$.
\par
Finally, we conclude using an energy argument:
\begin{alignat*}{2}
\notag
0 &= \displaystyle\int_\Omega (\bsbeta\cdot\nabla w +\mu w)J_p(w)
\\
\tag{by~\eqref{eq:beta_identity2} and~\eqref{eq:zero_trace}}
&=  \|w\|^{2-p}_p\left[\displaystyle\int_\Omega |w|^p\Big(\mu-{1\over p}
\div(\bsbeta)\Big)+
\int_{\partial\Omega_+}{\bsbeta\cdot\bsn}\,|w|^p\right]
\\
\tag{by~\eqref{eq:beta-mu}}
&\geq \mu_0\|w\|_p^2+\|w\|^{2-p}_p
\int_{\partial\Omega_+}{\bsbeta\cdot\bsn}\,|w|^p
\\
\notag
&\geq  \mu_0\|w\|_p^2\,.
\end{alignat*}
Hence $w=0$.
\par
On the other hand, the case when $\mu\equiv 0$ and $\div\bsbeta\equiv 0$ is straightforward (under Assumption~\ref{ass:omega-filling}) since $\bsbeta\cdot\nabla w=0$ implies
\begin{alignat}{2}
\tag{by Lemma~\ref{lem:Poincare_Friedrichs}} 
0 = \|\bsbeta\cdot\nabla w\|_p\geq {1\over \CPF}\|w\|_p\,\,.
\end{alignat}
\hfill%
$\ensuremath{_\blacksquare}$
\par
We are left with a proof of the statement~$\widetilde J_p(w) \in \Gq$:
\begin{lemma}[Regularity of $|w|^{p-1} \sign w$]\label{lem:beta_nabla}
Let $\mu,\bsbeta\in L^\infty(\Omega)$ and $w\in L^p(\Omega)$
satisfy the homogeneous advection--reaction equation
$$
\bsbeta\cdot\nabla w +\mu w= 0 \qquad \hbox{in } L^p(\Omega).
$$
Then the function $\widetilde J_p(w):=|w|^{p-1}\sign(w)\in L^q(\Omega)$ satisfies:
$$
\bsbeta\cdot\nabla\widetilde J_p(w) \in L^q(\Omega).
$$
\end{lemma}
\begin{proof}
First observe that $\widetilde J_p(w)$ has a G\^ateaux derivative in the direction $\bsbeta\cdot\nabla w$. Indeed,
\begin{alignat*}{2}
{\widetilde J}_p'(w)[\bsbeta\cdot\nabla w]
&= 
\displaystyle\lim_{t\to0}{\widetilde J_p(w+t\bsbeta\cdot\nabla w)-\widetilde J_p(w)\over t}
\\
&= \displaystyle\lim_{t\to0}{\widetilde J_p(w-t\mu w)-\widetilde J_p(w)\over t}
\\
&= \left(\displaystyle\lim_{t\to0}{|1-t\mu|^{p-2}(1-t\mu)-1\over t}\right)|w|^{p-1}\sign(w)
\\
&= -(p-1)\,\mu\,|w|^{p-1}\sign(w).
\end{alignat*}
Hence, ${\widetilde J}_p'(w)[\bsbeta\cdot\nabla w]\in L^q(\Omega)$. The conclusion of the lemma follows from the identity: 
$$
\bsbeta\cdot\nabla\widetilde J_p(w)
={\widetilde J}_p'(w)[\bsbeta\cdot\nabla w]\qquad \hbox{a.e. in } \Omega,
$$
which is straightforward to verify.
\end{proof}
%

\subsection{Proof of Proposition~\ref{prop:advContPair}}
\label{sec:advContPairProof}
We construct explicitly a Fortin operator $\Pi:\mbbV\to\mbbV_m$ satisfying Assumption~\ref{assumpt:Fortin}. We note that this 1-D proof is similar to the 1-D version of the proof of~\cite[Lemma~4.20, p.~190]{ErnGueBOOK2004}.
\par
Let $-1=x_0<x_1<\dots<x_n=1$ be the set of nodes defining the partition $\mcT_n$. Over each element $T_j=(x_{j-1},x_j)\in \mcT_n$ we define $\Pi$ to be the linear interpolant $\Pi_1$ plus a quadratic bubble, i.e.,
$$
\Pi(v)\Big|_{T_j}=\Pi_1(v)\Big|_{T_j} + \alpha_j Q_j(x) \in \mbbP^2(T_j), \qquad\forall v\in \mbbV,
$$
where $\Pi_1(v)\Big|_{T_j}= |T_{j-1}|^{-1}\big(v(x_{j-1})(x_j-x) + v(x_j)(x-x_{j-1})\big)$ and $Q_j(x)=(x-x_{j-1})(x-x_j)$. The coefficient $\alpha_j$ multiplying the bubble $Q_j(x)$ is selected in order to fulfill the equation:
\begin{alignat}{2}\label{eq:element_Pi}
\int_{T_j}\Pi(v)=\int_{T_j}v\,.
\end{alignat}
Observe that $\Pi(v)\in \mbbP^2_{\mathrm{cont},0,\{1\}}(\mcT_n)\subseteq \mbbP^k_{\mathrm{cont},0,\{1\}}(\mcT_n)$ since $k\ge 2$, and for all $w_n\in \mbbU_n$ we have: 
\begin{alignat*}{2}
\tag{by integration by parts}
b(w_n,\Pi(v)) & = \sum_{j=1}^n\int_{T_j}w_n'\Pi(v) - w_n\Pi(v)\Big|_{x_{j-1}}^{x_j}  \\
\tag{since $w_n\in\mbbP^1(T_j)$}
& = \sum_{j=1}^nw_n'\int_{T_j}\Pi(v) - w_n\Pi(v)\Big|_{x_{j-1}}^{x_j}\\
\tag{by interpolation and~\eqref{eq:element_Pi}}
& = \sum_{j=1}^nw_n'\int_{T_j}v - w_nv\Big|_{x_{j-1}}^{x_j}\\
\tag{by integration by parts}
& = b(w_n,v)\,.
\end{alignat*}
Hence, the requirement~\eqref{eq:Fortin_c} is satisfied. Now we recall that $\|(\cdot)\|_\mbbV:=\|(\cdot)'\|_q\,$. Therefore to obtain the requirement~\eqref{eq:Fortin_a} 
(i.e.~the boundedness of the operator $\Pi$), we note that on each element:
\begin{alignat*}{2}
|\alpha_j| & \leq  {6\over|T_j|^3}\int_{T_j}|v-\Pi_1(v)|
\leq {6\over|T_j|^{3-{1\over p}}}\|v-\Pi_1(v)\|_q
\leq {6\over|T_j|^{2-{1\over p}}}\|v'-\Pi_1(v)'\|_q\\
 \|\Pi_1(v)'\|_q & = {|v(x_j)-v(x_{j-1})|\over |T_j|^{1-{1\over q}}}={1\over |T_j|^{1-{1\over q}}}\left|\int_{T_j}v'\right|
 \leq \|v'\|_q\\
 \|Q_j'\|_q & = {|T_j|^{1+{1\over q}}\over (q+1)^{1\over q}}.
\end{alignat*}
Thus, on each element (and therefore globally) we have:
\begin{alignat*}{2}
\|\Pi(v)'\|_q & \leq \|\Pi_1(v)'\|_q+|\alpha_j|\|Q_j'\|_q\\
& \leq \|v'\|_q + C_q \|v'-\Pi_1(v)'\|_q\\
& \leq  (1+2C_q)\|v'\|_q,
\end{alignat*}
where the constant $C_q=6/(q+1)^{1\over q}$ is mesh-independent.%
\hfill%
$\ensuremath{_\blacksquare}$

\section*{Acknowledgements}
\addcontentsline{toc}{section}{Acknowledgements}
IM and KvdZ thank Leszek Demkowicz, Jay Gopalakrishnan, Paul Houston, Weifeng Qui and Sarah Roggendorf for helpful discussions. 
The work by IM was done in the framework of Chilean FONDECYT research project~\#1160774. IM was also partially supported by the European Union’s Horizon 2020, research and innovation program under the Marie Sklodowska-Curie grant agreement No~777778. 
MT and KvdZ are grateful for the support provided by the London Mathematical Society (LMS) Undergraduate Research Bursary Grant \emph{``Advanced discontinuous discretisation techniques for multiscale partial differential equations'' }17-18~103, and thank Donald Brown for his contributions. KvdZ also thanks the support provided by the Royal Society International Exchanges Scheme / Kan Tong Po Visiting Fellowship Programme, and the above FONDECYT project.
\small
\bibliography{BibFile_OAP}

\begin{thebibliography}{10}

\bibitem{AzePhD1996}
{\sc P.~Az\'erad}, {\em Analyse des \'equations de {Navier}--{Stokes} en bassin
  peu profond et de l'\'equation de transport}, PhD thesis, Universit\'e de
  Neuch\^atel, Neuch\^atel, Swiss, Jan 1996.

\bibitem{AzePouCR1996}
{\sc P.~Az\'erad and J.~Pousin}, {\em In\'egalit\'e de {P}oincar\'e courbe pour
  le traitement variationnel de l'\'equation de transport}, C.~R. Math. Acad.
  Sci. Paris, 322 (1996), pp.~721--727.

\bibitem{BarLerNedCPDE1979}
{\sc C.~Bardos, A.~Y. Leroux, and J.-C. N\'ed\'elec}, {\em First order
  quasilinear equations with boundary conditions}, Comm. Partial Differential
  Equations, 4 (1979), pp.~1017--1034.

\bibitem{BeiRDM1987}
{\sc H.~{Beir\~{a}o~Da~Veiga}}, {\em Existence results in {Sobolev} spaces for
  a stationary transport equation}, Ric.~Mat., XXXVI (1987), pp.~173--184.

\bibitem{BeiRSMUP1988}
\leavevmode\vrule height 2pt depth -1.6pt width 23pt, {\em Boundary-value
  problems for a class of first order partial differential equations in
  {Sobolev} spaces and applications to the {Euler} flow}, Rend. Sem. Mat. Univ.
  Padova, 79 (1988), pp.~247--273.

\bibitem{BreScoBOOK2008}
{\sc S.~C. Brenner and L.~R. Scott}, {\em The Mathematical Theory of Finite
  Element Methods}, vol.~15 of Texts in Applied Mathematics, Springer, Berlin,
  3rd~ed., 2008.

\bibitem{BroDahSteMOC2018}
{\sc D.~Broersen, W.~Dahmen, and R.~P. Stevenson}, {\em On the stability of
  {DPG} formulations of transport equations}, Math. Comp., 87 (2018),
  pp.~1051--1082.

\bibitem{BuiDemGhaMOC2013}
{\sc T.~{Bui-Thanh}, L.~Demkowicz, and O.~Ghattas}, {\em Constructively
  well-posed approximation methods with unity inf-sup and continuity constants
  for partial differential equations}, Math. Comp., 82 (2013), pp.~1923--1952.

\bibitem{CanCR2017}
{\sc P.~Cantin}, {\em Well-posedness of the scalar and the vector
  advection--reaction problems in {Banach} graph spaces}, C.~R. Math. Acad.
  Sci. Paris, 355 (2017), pp.~892--902.

\bibitem{CanHeuHAL2018}
{\sc P.~Cantin and N.~Heuer}, {\em A {DPG} framework for strongly monotone
  operators}.
\newblock \href{https://hal.archives-ouvertes.fr/hal-01690281}{hal-01690281},
  Jan 2018.

\bibitem{CarBriHelWri2017}
{\sc C.~Carstensen, P.~Bringmann, F.~Hellwig, and P.~Wriggers}, {\em Nonlinear
  discontinuous {Petrov--Galerkin} methods}.
\newblock arXiv:1710.00529v1 [math.NA], 2017.

\bibitem{ChaDemMosCF2014}
{\sc J.~Chan, L.~Demkowicz, and R.~Moser}, {\em A {DPG} method for steady
  viscous compressible flow}, Comput. \& Fluids, 98 (2014), pp.~69--90.

\bibitem{ChaEvaQiuCAMWA2014}
{\sc J.~Chan, J.~A. Evans, and W.~Qiu}, {\em A dual {P}etrov--{G}alerkin finite
  element method for the convection--diffusion equation}, Comput. Math. Appl.,
  68 (2014), pp.~1513--1529.

\bibitem{CohDahWelM2AN2012}
{\sc A.~Cohen, W.~Dahmen, and G.~Welper}, {\em Adaptivity and variational
  stabilization for convection-diffusion equations}, M2AN Math. Model. Numer.
  Anal., 46 (2012), pp.~1247--1273.

\bibitem{DahHuaSchWelSINUM2012}
{\sc W.~Dahmen, C.~Huang, C.~Schwab, and G.~Welper}, {\em Adaptive
  {P}etrov--{G}alerkin methods for first order transport equations}, SIAM J.
  Numer. Anal., 50 (2012), pp.~2420--2445.

\bibitem{DauLioBOOK1992}
{\sc R.~Dautray and J.-L. Lions}, {\em Mathematical Analysis and Numerical
  Methods for Science and Technology. Vol. 5: Evolution Problems~{I}},
  Springer-Verlag, Berlin, 1992.

\bibitem{DauLioBOOK1993}
\leavevmode\vrule height 2pt depth -1.6pt width 23pt, {\em Mathematical
  Analysis and Numerical Methods for Science and Technology. Vol. 6: Evolution
  Problems~{II}}, Springer-Verlag, Berlin, 1993.

\bibitem{DemGopCMAME2010}
{\sc L.~Demkowicz and J.~Gopalakrishnan}, {\em A class of discontinuous
  {P}etrov--{G}alerkin methods. {P}art {I}. {T}he transport equation}, Comput.
  Methods Appl. Mech. Engrg., 199 (2010), pp.~1558--1572.

\bibitem{DemGopNMPDE2011}
\leavevmode\vrule height 2pt depth -1.6pt width 23pt, {\em A class of
  discontinuous {P}etrov--{G}alerkin methods. {II}. {O}ptimal test functions},
  Numer. Methods Partial Differential Equations, 27 (2011), pp.~70--105.

\bibitem{DemGopBOOK-CH2014}
\leavevmode\vrule height 2pt depth -1.6pt width 23pt, {\em An overview of the
  discontinuous {Petrov Galerkin} method}, in Recent Developments in
  Discontinuous Galerkin Finite Element Methods for Partial Differential
  Equations: 2012 John H Barrett Memorial Lectures, X.~Feng, O.~Karakashian,
  and Y.~Xing, eds., vol.~157 of The IMA Volumes in Mathematics and its
  Applications, Springer, Cham, 2014, pp.~149--180.

\bibitem{DipErnBOOK2012}
{\sc D.~A. {Di~Pietro} and A.~Ern}, {\em Mathematical Aspects of Discontinuous
  Galerkin Methods}, vol.~69 of Math\'{e}matiques et Applications, Springer,
  Berlin, 2012.

\bibitem{ErnGueBOOK2004}
{\sc A.~Ern and J.-L. Guermond}, {\em Theory and Practice of Finite Element
  Methods}, vol.~159 of Applied Mathematical Sciences, Springer-Verlag, New
  York, 2004.

\bibitem{ErnGueCR2016}
\leavevmode\vrule height 2pt depth -1.6pt width 23pt, {\em A converse to
  {F}ortin's {L}emma in {B}anach spaces}, C.~R. Math. Acad. Sci. Paris, 354
  (2016), pp.~1092--1095.

\bibitem{ErnGueM2AN2017}
\leavevmode\vrule height 2pt depth -1.6pt width 23pt, {\em Finite element
  quasi-interpolation and best approximation}, M2AN Math. Model. Numer. Anal.,
  51 (2017), pp.~1367--1385.

\bibitem{GirTarCR2010}
{\sc V.~Girault and L.~Tartar}, {\em {$L^p$} and {$W^{1,p}$} regularity of the
  solution of a steady transport equation}, C.~R. Math. Acad. Sci. Paris, 348
  (2010), pp.~885--890.

\bibitem{GopMonSepCAMWA2015}
{\sc J.~Gopalakrishnan, P.~Monk, and P.~Sep\'ulveda}, {\em A tent pitching
  scheme motivated by {Friedrichs} theory}, Comput. Math. Appl., 70 (2015),
  pp.~1114--1135.

\bibitem{GopQiuMOC2014}
{\sc J.~Gopalakrishnan and W.~Qiu}, {\em An analysis of the practical {DPG}
  method}, Math. Comp., 83 (2014), pp.~537--552.

\bibitem{GueM2AN1999}
{\sc J.-L. Guermond}, {\em Stabilization of {Galerkin} approximations of
  transport equations by subgrid modeling}, M2AN Math. Model. Numer. Anal., 33
  (1999), pp.~1293--1316.

\bibitem{GueSINUM2004}
{\sc J.~L. Guermond}, {\em A finite element technique for solving first-order
  {PDEs} in {$L^p$}}, SIAM J. Numer. Anal., 42 (2004), pp.~714--737.

\bibitem{HolRisBOOK2015}
{\sc H.~Holden and N.~H. Risebro}, {\em Front Tracking for Hyperbolic
  Conservation Laws}, vol.~152 of Applied Mathematical Sciences, Springer,
  Berlin, 2nd~ed., 2015.

\bibitem{LavSINUM1989}
{\sc J.~E. Lavery}, {\em Solution of steady-state one-dimensional conservation
  laws by mathematical programming}, SIAM J. Numer. Anal., 26 (1989),
  pp.~1081--1089.

\bibitem{MugZeeARXIV2018}
{\sc I.~Muga and K.~G. van~der Zee}, {\em Discretization of linear problems in
  {Banach} spaces: {R}esidual minimization, nonlinear {P}etrov--{G}alerkin, and
  monotone mixed methods}.
\newblock \mbox{\href{http://arxiv.org/abs/1511.04400}{arXiv:1511.04400v3
  [math.NA]}}, 2018.

\bibitem{PiaARXIV2016}
{\sc T.~Piasecki}, {\em Steady transport equation in {Sobolev-Slobodetskii}
  spaces}.
\newblock {arXiv:1512.00868v2} [math.AP], 2016.

\bibitem{SafTasEJA1999}
{\sc E.~B. Saff and S.~Tashev}, {\em Gibbs phenomenon for best {$L_p$}
  approximation by polygonal lines}, East J. Approx., 5 (1999), pp.~235--251.

\bibitem{SteNM2015}
{\sc A.~Stern}, {\em Banach space projections and {Petrov--Galerkin}
  estimates}, Numer. Math., 130 (2015), pp.~125--133.

\bibitem{VerBOOK2013}
{\sc R.~Verf{\"{u}}rth}, {\em A Posteriori Error Estimation Techniques for
  Finite Element Methods}, Numerical Mathematics and Scientific Computation,
  Oxford Science Publications, Oxford, 2013.

\end{thebibliography}
\bibliographystyle{siam}
\end{document}
%